\documentclass[reqno,10pt,centertags]{amsart} 
\usepackage{amsmath,amsthm,amscd,amssymb,latexsym,esint,upref,stmaryrd,enumerate,xcolor,verbatim,yfonts,multicol}
\calclayout
\usepackage{color,float}
\usepackage{cancel}
\usepackage{hyperref}
\usepackage{cleveref}

\usepackage{tikz}
\usetikzlibrary{calc,decorations.markings}
\usepackage{pgfplots}
\pgfplotsset{compat=1.15}

\usepackage{graphicx}
\usepackage{enumitem}

\newcommand*{\mailto}[1]{\href{mailto:#1}{\nolinkurl{#1}}}
\newcommand{\arxiv}[1]{\href{http://arxiv.org/abs/#1}{arXiv:#1}}
\newcommand{\doi}[1]{\href{http://dx.doi.org/#1}{DOI:#1}}

\newcommand{\R}{{\mathbb R}}
\newcommand{\N}{{\mathbb N}}
\newcommand{\Z}{{\mathbb Z}}
\newcommand{\C}{{\mathbb C}}

\newcommand{\bbC}{{\mathbb{C}}}

\newcommand{\bbN}{{\mathbb{N}}}

\newcommand{\bbR}{{\mathbb{R}}}

\newcommand{\con}{\mathfrak{p}}
\newcommand{\cona}{\mathfrak{p}_a}
\newcommand{\conb}{\mathfrak{p}_b}
\newcommand{\coninf}{\mathfrak{p}_\infty}

\newcommand{\beq}{\begin{align}}
\newcommand{\enq}{\end{align}}

\renewcommand{\a}{\alpha}
\renewcommand{\b}{\beta}

\newcommand{\e}{\varepsilon}

\newcommand{\z}{\zeta}

\DeclareMathOperator{\dist}{dist}

\DeclareMathOperator{\dom}{dom}

\DeclareMathOperator{\Ai}{Ai}
\DeclareMathOperator{\Bi}{Bi}

\newcommand{\Arg}{\text{\rm Arg}}

\renewcommand{\Re}{\text{\rm Re}}

\renewcommand{\ln}{\text{\rm ln}}

\newcommand{\no}{\notag}
\newcommand{\lb}{\label}
\newcommand{\f}{\frac}

\newcommand{\wti}{\widetilde}

\newcommand{\oh}{o}
 
\newcommand{\dott}{\,\cdot\,}

\newcommand{\bi}{\bibitem}

\let\geq\geqslant
\let\leq\leqslant

\newcommand{\al}{\a}

\newcommand{\eps}{\varepsilon}

\newcommand{\Lr}{{L^2((a,b);rdx)}} 

\newcommand{\ACl}{{AC_{loc}((a,b))}}
\newcommand{\Ll}{{L^1_{loc}((a,b);dx)}}

\makeatletter
\def\theequation{\@arabic\c@equation}
\allowdisplaybreaks 
\numberwithin{equation}{section}

\newtheorem{theorem}{Theorem}[section]
\newtheorem{proposition}[theorem]{Proposition}
\newtheorem{lemma}[theorem]{Lemma}
\newtheorem{corollary}[theorem]{Corollary}
\newtheorem{definition}[theorem]{Definition}
\newtheorem{hypothesis}[theorem]{Hypothesis}

\newtheorem{problem}[theorem]{Open Problem} 

\theoremstyle{remark}
\newenvironment{remark}[1][]{\refstepcounter{theorem}\par\medskip\noindent\textit{Remark~$\theexample. #1$} \rmfamily}{{\ }\hfill $\diamond$ \vspace{6pt}}%Define remark environment with diamond at the end

\begin{document}

\title[Sturm--Liouville operators with Schatten $p$-class resolvents]{Complex analytic theory of Sturm--Liouville operators with Schatten $p$-class resolvents} 

\author[G.\ Fucci]{Guglielmo Fucci}
\address{Department of Mathematics, 
East Carolina University, 331 Austin Building, East Fifth St.,
Greenville, NC 27858-4353, USA}
\email{\mailto{fuccig@ecu.edu}}
%\email{fuccig@ecu.edu}
\urladdr{\url{http://myweb.ecu.edu/fuccig/}}
%\urladdr{http://myweb.ecu.edu/fuccig/}

\author[M. Piorkowski]{Mateusz Piorkowski}
\address{Department of Mathematics, KTH Royal Institute of Technology, Stockholm, Sweden}
\email{\mailto{mateuszp@kth.se}}
\urladdr{\url{https://sites.google.com/view/mateuszpiorkowski/home?pli=1}}

\author[J.\ Stanfill]{Jonathan Stanfill}
\address{Division of Geodetic Science, School of Earth Sciences, The Ohio State University \\
275 Mendenhall Laboratory, 125 South Oval Mall, Columbus, OH 43210, USA}
\email{\mailto{stanfill.13@osu.edu}}
%\email{stanfill.13@osu.edu}
\urladdr{\url{https://u.osu.edu/stanfill-13/}}
%\urladdr{https://u.osu.edu/stanfill-13/}

%\dedicatory{}

\date{\today}
%\thanks{} 
%\thanks{Appeared in {\it .}
\@namedef{subjclassname@2020}{\textup{2020} Mathematics Subject Classification}
\subjclass[2020]{Primary: 34B24, 34E05, 47A10, 47B10. Secondary: 34B27, 34L40.}
\keywords{Sturm--Liouville operators, Liouville--Green (WKB) asymptotics, spectral $\zeta$-functions.}

%%%%%%%%%%%%%%%%%%%%%%%%%%%%%%%

\begin{abstract}
    We use the theory of entire functions of finite order to prove a universal spectral dependence of the blowup/decay rate of solutions of the Sturm--Liouville eigenvalue equation for problems with Schatten $p$-class resolvents. The general form of the asymptotics turns out to depend exclusively on the largest \emph{integer} $\con$ such that the underlying resolvents fail to be in the Schatten $\con$-class.
    
    We then use the above result to construct a characteristic function of minimal order for Sturm--Liouville problems with Schatten $p$-class resolvents. This immediately yields contour integral representations of spectral $\z$-functions that were previously only known for quasi-regular problems (except for a few examples). We also demonstrate how our methods lead to new results in connection to  important classic topics of Liouville--Green (or WKB) asymptotics and the approximation of the spectrum of singular problems via underlying truncated regular problems. All our applications are accompanied by illustrative examples, including the Airy differential equation, harmonic oscillator (and general power potentials), and the Laguerre differential equation.  
\end{abstract}
%%%%%%%%%%%%%%%%%%%%%%%%%%%%%%%

\maketitle

%%%%%%%%%%%%%%%%%%%%%%%%%%%%%%%
%%%%%%%%%%%%%%%%%%%%%%%%%%%%%%%
\section{Introduction}
%%%%%%%%%%%%%%%%%%%%%%%%%%%%%%%
%%%%%%%%%%%%%%%%%%%%%%%%%%%%%%%

The main motivation for the current work is spectral $\zeta$-functions, in particular, the process of analytically continuing their sum definition to the rest of the complex plane. The properties revealed in this process, such as the location and order of poles as well as trivial zeros and special values, are very important in many areas of mathematics and physics \cite{Em12,Em94,Ki02}. The starting point for this analysis is the contour integral representation for the spectral $\zeta$-function recently generalized in \cite{FPS25a}. A key part of such formulas is the \emph{characteristic function}, which is an entire function vanishing exactly at the eigenvalues (see Def.~\ref{defchar}). While not unique, in the context of integral representations it is essential that the characteristic function is of finite (\emph{ideally minimal}) order as an entire function (see \cite[Lect.~1]{Levin1996}). This is only possible if the eigenvalues do not grow \emph{too slowly}; that is, they must satisfy at least $\lambda_n \gtrsim n^{1/\kappa}$ for some $\kappa > 0$. Equivalently, resolvents of self-adjoint realizations need to be in some Schatten $p$-class, which we will assume throughout.

One of our main results is the construction of characteristic functions of minimal order, see Theorem \ref{Thrm3.4} below. The key technical ingredient in our analysis is the establishment of precise formulas for the asymptotics of principal and nonprincipal solutions at the singular endpoints found in Theorem \ref{TheoremUV}. These formulas are \emph{universal}, in the sense that their general form depends only on the largest \emph{integer} $\con$ such that the underlying resolvents fail to be in the Schatten $\con$-class. No additional smoothness of the coefficients of the Sturm--Liouville expression, except for the standard $L^1_{loc}$-assumptions (see Hypothesis \ref{h1}), is required for this result to hold. In Section \ref{Sect:Applications} we demonstrate how our asymptotic formulas for the (non)principal solutions have immediate consequences to other topics of interest: Liouville--Green (or WKB) asymptotics and related integral formulas (see Sect.~\ref{Sect:LG}), and the approximation of the spectrum of singular problems in terms of regular problems on a truncated interval (see Sect.~\ref{Sect:EigenConv}).

Our analysis relies on (mostly elementary) results from the theory of entire function, in particular the Hadamard factorization for entire functions of finite order. In fact, the aforementioned asymptotics of (non)principal solutions inherit their form directly from the exponential in Weierstrass' elementary factors $E(z, p) = (1-z)\exp(\sum_{\ell=1}^p \frac{z^\ell}{\ell})$. For a general reference to the theory of entire functions see \cite{Con73,Levin1996,Nev07}.

Below we start with a brief overview of spectral $\zeta$-functions and their integral representations. In particular, we focus on the role of the characteristic function and illustrate the importance of it being of minimal order. We then summarize our main results in Section \ref{Sect:Results} (the proofs are found in later sections).

\subsection{Motivation via spectral \texorpdfstring{$\zeta$}{zeta}-functions}
The motivation for this work arises from the study of spectral $\zeta$-functions. In particular the question of how one can construct an integral representation for the spectral $\zeta$-function associated with self-adjoint extensions of singular Sturm--Liouville operators defined in the \emph{largest possible} region of the complex plane. Formally, the spectral $\zeta$-function for some self-adjoint operator $T$ with a purely discrete spectrum $\sigma(T)$ is given by
\begin{align}
    \zeta(s;T) = \sum_{\underset{\lambda_n\neq 0}{n\in\bbN}} \lambda_n^{-s}, \quad \sigma(T) = \lbrace \lambda_n \rbrace_{n=1}^\infty,
\end{align}
where convergence is only guaranteed for $\Re(s)$ large enough. In our recent paper \cite{FPS25a}, we construct integral representations of $\zeta$-functions associated with general sequences of complex numbers $\{\mu_{n}\}_{n\in \bbN}$ satisfying (see \cite[Def. 2.1]{FPS25a}):\\[1mm]
$(i)$ $0<|\mu_1|\leq|\mu_2|\leq \dots \to\infty$,\\[1mm]
    $(ii)$ the exponent of convergence $\kappa=\inf\big\{\rho > 0 \, \colon \sum_{n\in\bbN} |\mu_{n}|^{-\rho} < \infty \big\}$ is finite,\\[1mm]
    $(iii)$  there exists $\varepsilon>0$ and $\Psi\in[-\pi,\pi)$ such that all $\mu_n\in \mathcal{I}_\varepsilon=\left\{z\in\C : \Arg(z) \not \in (\Psi-\varepsilon, \Psi+\varepsilon)\right\}$.\\[1mm]
(The exponent of convergence can be computed via the formula $\kappa = \limsup_{n \to \infty} (\ln \, n)/(\ln \, |\mu_n|)$.)

Sturm--Liouville expressions that have a \emph{finite regularization index} at both endpoints \cite{PS24} (this includes the regular and quasi-regular cases) have self-adjoint extensions which are bounded from below, possess discrete spectrum, and satisfy Weyl asymptotics (i.e.,~$\lambda_n \propto n^2$). In each of these cases conditions $(i)$ and $(iii)$ are trivially staisfied (for any choice of $\Psi\in(-\pi,0)\cup(0,\pi)$), and condition $(ii)$ holds with exponent of convergence $\kappa=1/2$. Therefore one can conclude that in these cases the spectral $\z$-function is well-defined at least for $\Re(s) > 1/2$.  Indeed, the analysis of the spectral $\z$-function and its analytic continuation for regular and quasi-regular problems has been thoroughly investigated in \cite{FGKS21} and \cite{FPS25}, respectively.     

In the case of general singular Sturm–Liouville operators, the situation is more subtle because the spectrum could have, for instance, a continuous part. In addition, 
even when the spectrum is discrete, there are some cases in which $\kappa$ may not be finite due to the spectrum growing too slowly. For instance, it was shown in \cite{GP78} that the eigenvalues for the Schr\"odinger problem endowed with a logarithmic (plus Bessel) potential on the half-line has eigenvalues growing like $\ln\, n$ as $n\to\infty$ (while the solutions to the problem do not have finite growth order in the spectral parameter).
To ensure that the spectral $\zeta$-function is well-defined in the singular case, we must then introduce some additional assumptions:
\begin{hypothesis}\label{h3a}
Hypothesis \ref{h2} holds and the self-adjoint extensions $T_A$ in Theorem \ref{extensions} have only discrete spectrum $\sigma(T_{A})=\{\lambda_{A,n}\}_{n\in \bbN}$ with a finite exponent of convergence 
\begin{align}
    \kappa = \inf\Big\{\rho > 0 \, \colon \sum_{\underset{\lambda_{A,n}\neq 0}{n\in\bbN}} |\lambda_{A,n}|^{-\rho} < \infty \Big\},
\end{align}
where the eigenvalues are counted according to their multiplicity.
\end{hypothesis}
Note that this hypothesis is equivalent to the resolvents being in the Schatten $p$-class with $p > \kappa$.
Clearly it guarantees that $\sigma(T_{A})$ satisfies conditions $(i)$, $(ii)$, and $(iii)$ above (once again, for any choice of $\Psi\in(-\pi,0)\cup(0,\pi)$) and, therefore, the \textit{spectral $\zeta$-function} associated with the self-adjoint extension $T_{A}$ given by the infinite sum
\begin{align}\lb{2.65}
\zeta(s;T_{A}):=\sum_{\underset{\lambda_n\neq 0}{n\in\bbN}} \lambda_{A,n}^{-s},\quad \Re(s)>\kappa,
\end{align}
converges in its domain of definition. In case of negative eigenvalues $\lambda_{A,n}$, we make the convention that $\lambda_{A,n}^s = e^{-\pi i s}(-\lambda_{A,n})^s$, in agreement with \cite{FGKS21,FPS25,KM03,KM04}.

In order to study the properties of the analytic continuation of $\zeta(s; T_{A})$ (e.g.~its pole structure or branch points), it is convenient to obtain an integral representation valid for $\Re(s)>\kappa$ which can then be analytically continued to a larger region of the complex plane (see e.g. \cite{Ki02}). 
To construct a suitable integral representation we introduce the notion of characteristic function: 
\begin{definition}\label{defchar}
We call any entire function $F_A(z)$ with finite order, $\rho\geq\kappa$, and zeros exactly at $\sigma(T_{A})=\{\lambda_{A,n}\}_{n\in \bbN}$, counting multiplicities, a \textup{characteristic function} associated with $\sigma(T_{A})$.
\end{definition}
Note that existence of such characteristic functions is guaranteed by the Hadamard factorization theorem (see the proof of Thm.~\ref{Thrm3.4} below).

We recall here that an entire function $G$ is of finite order $\rho \geq 0$ if and only if
\begin{align}
    \exp(r^{\rho-\varepsilon}) \leq \max_{|z| = r} |G(z)| \leq \exp(r^{\rho+\varepsilon})
\end{align}
holds for any $\varepsilon > 0$ and $r \geq r_\varepsilon$. Assume now that $G(z)$ has infinitely many zeros $\lbrace \mu_n \rbrace_{n \in \bbN}$, and denote by $\kappa$ their exponent of convergence. Then we always have the inequality $\rho \geq \kappa$ (\cite[Ch.~3.2, Thm.~2]{Levin1996}). We say that $G(z)$ is of \emph{minimal order} if $\rho = \kappa$. By the previous assertion, this means that $G(z)$ is an  entire function with zero set $\lbrace \mu_n \rbrace_{n \in \bbN}$ and the \emph{slowest possible} growth at infinity. A closely related concept is the \emph{rank} $\con$ of a sequence of complex numbers defined via \cite[P. 284]{Con73}
\begin{align}
    \con = \min\Big\lbrace p \in \mathbb{N}_0=\bbN\cup\{0\} \ \colon \ \sum_{\underset{\mu_n\neq 0}{n\in\bbN}} |\mu_n|^{-(p+1)} < \infty \Big\rbrace.
\end{align}
Note that $\con$ is, by assumption, an integer. In all of the examples that we consider we will have $\con = \lfloor \kappa \rfloor$. While, in cases where $\kappa \not \in \mathbb N$, this equality will hold automatically, see \cite[Ch.~5.1, Thm.~1]{Levin1996}, for integer $\kappa$ we can in general only conclude that $\con \in \lbrace \kappa - 1, \kappa\rbrace$. Illustrative examples for $\con = \kappa \in \mathbb N$ are sequences $\mu_n \propto n^{\frac{1}{\kappa}}$, while for $\con = \kappa - 1 \in \mathbb N_0$ we could chose $\mu_n \propto n^{\frac{1}{\kappa}} \ln^{\frac{2}{\kappa}}(n)$. See Remark \ref{RemarkPhiKappa} for some consequences related to this case distinction.

According to \cite[Thm.~2.6]{FPS25a}, once a characteristic function $F_A(z)$ of order $\rho$ associated with $\sigma(T_{A})$ is available, one obtains the following integral representation for the spectral $\zeta$-function of $T_{A}$ valid for $\Re (s) > \rho$:
\begin{theorem}\label{Thrm3.4}
    Let Hypothesis \ref{h3a} be satisfied. If $F_{A}(z)$ is a characteristic function of order $\rho$ associated with the spectrum $\sigma(T_A)$, then for $\Re(s)>\rho\geq\kappa$,
  \begin{align}\label{3.2}
\zeta(s;T_A)&=\frac{1}{2\pi i}\int_\gamma dz \, z^{-s} \Big[\frac{d}{dz} \ln \, F_A(z)-\frac{m_0}{z}\Big]\\
&=e^{is(\pi-\Psi)}\frac{\sin(\pi s)}{\pi}\int_{R}^{\infty}dt\,t^{-s}\Big[\frac{d}{dt}\ln\,F_A\left(te^{i\Psi}\right)-\frac{m_0}{te^{i\Psi}}\Big]-\frac{1}{2\pi i}\int_{C_R}dz\,z^{-s}\Big[\frac{d}{dz}\ln\,F_A\left(z\right)-\frac{m_0}{z}\Big],\notag
\end{align}
where $m_0$ represents the multiplicity of the zero eigenvalue, $\rho$ is the order of $F_A(z)$, $C_R$ is a clockwise circle with radius $0<R<\min\{|\lambda_{A,n}|: \lambda_{A,n}\neq0\}$  parametrized via $z=R e^{i\theta}$ from $\theta=\Psi$ to $\theta=\Psi-2\pi$, $\Psi\in (\pi/2,\pi),$ and $\gamma$ is a counterclockwise contour which encloses the spectrum $\sigma(T_{A})$, dipping below, hence avoiding, the origin $($see \cite[Fig. 1]{FPS25}$)$. 
\end{theorem}
\begin{proof}
Since $\sigma(T_A)$ satisfies Hypothesis \ref{h3a} and has a finite exponent of convergence $\kappa\leq\rho$, we use \cite[Thm. 2.4]{FPS25a} to write $\zeta(s;T_A)=\frac{1}{2\pi i}\int_\gamma dz \, z^{-s} [\frac{d}{dz} \ln \, H_A(z)-\tfrac{m_0}{z}]$ for $\Re(s)>\kappa$,
where $H_{A}(z)$ denotes the characteristic function given by the Hadamard form \cite[Sect.~13.9 and 13.11]{Nev07}
\begin{equation}\label{c1}
H_A(z)=z^{m_0}\prod_{\underset{\lambda_n\neq 0}{n\in\bbN}}^\infty E\left(\frac{z}{\lambda_{A,n}},\con\right),\ \text{ with }\ 
  E\left(\frac{z}{\lambda_{A,n}},\con\right)=\left(1-\frac{z}{\lambda_{A,n}}\right)\exp\left[\sum_{j=1}^{\con}\frac{1}{j}\left(\frac{z}{\lambda_{A,n}}\right)^{j}\right].
\end{equation}
Any other entire function $F_{A}(z)$ of finite order $\rho$ and with the same zeroes as $H_{A}(z)$ in \eqref{c1} can be written as \cite[Sect.~13.9]{Nev07} $F_A(z)=e^{\omega(z)} H_A(z)$
for some polynomial $\omega(z)$, in which case $\rho=\max\{\deg(\omega), \kappa\}$.
Hence one can then write for $\Re(s)>\rho$ (since $z^{-s}\tfrac{d}{dz}\omega(z)$ is holomorphic in the interior of $\gamma$)
\begin{align}
\begin{split}
    \frac{1}{2\pi i} \int_\gamma dz  \, z^{-s} \Big[\frac{d}{dz} \ln \, F_A(z)-\frac{m_0}{z}\Big] &= \frac{1}{2\pi i} \int_\gamma dz  \, z^{-s} \Big[\frac{d}{dz} \ln \, H_A(z)-\frac{m_0}{z}\Big] + \frac{1}{2\pi i} \int_\gamma dz  \, z^{-s} \frac{d}{dz} \omega(z)
    \\
    &= \zeta(s;T_A).
\end{split}
\end{align}

The second equality in \eqref{3.2} is proven by applying \cite[Thm. 2.6]{FPS25a} in the region $\Re(s)>\rho$.
\end{proof}
It is clear from the above proof that if $F_{A}(z)$ has growth order $\rho\notin\mathbb{N}$, then it is automatically of minimal growth order, that is $\rho = \kappa$.  

As we have already mentioned earlier, the properties of the spectral $\z$-function, for instance the location and order of its poles, are very important and the integral representation \eqref{3.2} is a particularly suitable tool for this analysis. In fact, it provides a convenient starting point for the process of analytic continuation of the spectral $\z$-function in the complex plane (a process described thoroughly in, for example, \cite{FGKS21,FPS25,FPS25a}). 

To fully benefit from the integral representation  provided in \eqref{3.2} of Theorem \ref{Thrm3.4}, one would ideally want an explicit characteristic function $F_A(z)$ of \emph{minimal order}. In fact, if $F_{A}(z)$ is chosen to be of minimal growth order $\rho = \kappa$, then the representation \eqref{3.2} becomes valid in the largest possible region of the complex plane given by $\Re(s) > \kappa$, providing us with the optimal starting point for computing the analytic continuation to the region $\Re(s)  \leq \kappa$ as outlined in \cite{FPS25}. This observation is the main motivation for the present work.

While the construction of a characteristic function in the quasi-regular setting has been fully addressed in our previous work \cite{FPS25}, the general setting we consider here has not. As we will see below, the construction of characteristic functions of minimal growth order in the case of operators without trace-class resolvent (i.e., $\con \geq 1$)  proves to be much more difficult and interesting.

\begin{remark}
It turns out that the central parameter in our analysis will be $\con \in \mathbb N_0$ rather than the exponent of convergence $\kappa >0$ of $\sigma(T_A)$. Hence, most of our results are formulated under the assumption that resolvents $(T_A-zI)^{-1}$, $z \in \bbC \setminus \sigma(T_A)$ are in the Schatten $p$-class (also called $p$-Schatten--von Neumann class) with $p=\con+1$ but not $p = \con$. 
\end{remark}

\begin{remark}\label{RemarkJacobi}
    The case of Jacobi operators with resolvents in some Schatten $p$-class has been studied recently in \cite{SS26}. The authors obtained formulas for characteristic functions analogous to \eqref{FormF0} below, see \cite[Sect.~4.4]{SS26}. Just as here, these formulas rely on the Hadamard factorization theorem for entire functions of finite order. A crucial difference, however, lies in the notion of \emph{truncation}. In the case of Jacobi matrices one obtains a \emph{finite-dimensional} matrix, while in the Sturm--Liouville case the truncation to some $(c,d) \subset (a,b)$ leads to a regular problem which is, however, still \emph{infinite-dimensional}.
\end{remark}

\subsection{Summary of the main results}\label{Sect:Results}

We first define local exponents of convergence for truncated problems, then state our main theorem on the construction of characteristic functions of minimal order:

\begin{definition}\label{DefKappa}
Assume Hypothesis \ref{h3a} holds and that the exponent of convergence, $\kappa$, associated with $\sigma(T_A)$ is finite. We define the \textup{exponent of convergence near $a$} $($resp., $b$$)$ as the exponent of convergence, $\kappa_a$ $($resp., $\kappa_b$$)$, associated with the eigenvalues of self-adjoint realizations of $\tau|_{(a,c)}$ $($resp., $\tau|_{(c,b)}$$)$, $c\in(a,b)$ $($which will neither depend on $c$ nor the particular self-adjoint realization$)$. In a similar fashion we define $\cona, \conb \in \mathbb N_0$.
\end{definition}

Our main result on the construction of characteristic functions is stated next. A refined version with more `symmetric' requirements for the principal and nonprincipal solutions appears in Theorem \ref{TheoremChar}. 
\begin{theorem}\label{MainTheorem1}
    Let $\kappa$, $\cona$, $\conb$ be as in Definition \ref{DefKappa}. Then there exists a solution $\varphi_a(z,x)$ of $\tau f = zf$ which is principal at the endpoint $x = a$ and entire in $z$ of (minimal) order $\kappa_a$. Furthermore, there exists a solution $\theta_b(z,x)$ of $\tau f = zf$ which is nonprincipal at $x = b$ such that for all tuples $z_1, \dots, z_n \in \mathbb C$, and $w_1, \dots, w_n \in \mathbb C$ with $\sum_{j=1}^n z_j^\ell = \sum_{j=1}^n w_j^\ell$ for $\ell =1, \dots, \conb$ we have
    \begin{align}\label{TTTT}
        \lim_{x \uparrow b} \frac{\theta_b(z_1, x) \cdots  \theta_b(z_n,x)}{\theta_b(w_1, x) \cdots \theta_b(w_n ,x)} = 1.
    \end{align}
    With this choice, a characteristic function $F_0(z)$ of the Friedrichs realization of minimal order $\kappa = \max \lbrace \kappa_a, \kappa_b \rbrace$ is then given by the formula
    \begin{align}\label{FormF0}
        F_0(z) = \lim_{x \uparrow b} \frac{\varphi_a(z,x)}{\theta_b(z,x)},\quad z\in\bbC.
    \end{align}
\end{theorem}
Further conditions equivalent to $\varphi_a(z,x)$ being of minimal order are collected in Corollary \ref{Coriii}. 

While the condition on the principal (at $x = a$) solution $\varphi_a(z,x)$ might seem natural given that we are looking for a characteristic function of minimal order, the condition on the nonprincipal (at $x = b$) solution $\theta_b(z,x)$ in \eqref{TTTT} appears less motivated. It will follow naturally from the analysis in Section \ref{Sect::NonPrincipal} which is based on \emph{partial $\zeta$-values}, a central concept governing the behavior of (non)principal solutions at the respective endpoints. We turn to this topic next.

Consider the Sturm--Liouville differential expression $\tau|_{(c,d)}$ restricted to the interval $(c,d)$ with $a < c < d < b$. In this case, $\tau|_{(c,d)}$ will be regular at $c$ and $d$. We use the notation $\lambda_n(c,d)$ to denote the $n^{th}$ Dirichlet eigenvalue corresponding to $\tau|_{(c,d)}$. A crucial role in the present work will be played by the partial $\zeta$-values defined via
\begin{align}
    \zeta(\ell; (c,d)) = \sum_{n \in \bbN} \frac{1}{\lambda_n^\ell(c,d)}, \quad \ell \in \bbN
\end{align}
(the issue of zero eigenvalues will not be a problem in practice). Note that these are just the spectral $\zeta$-functions for the Dirichlet realization of $\tau|_{(c,d)}$ evaluated at positive integers -- see Proposition \ref{Prop:ExactZeta} for iterated integral formulas for these values. As $\tau|_{(c,d)}$ is regular at both endpoints, the eigenvalues grow quadratically, meaning that all partial $\zeta$-values are finite. However, as $\lambda_n(c,d) \downarrow \lambda_n(a,d)$ for $c \downarrow a$, see \cite[Ch.~10.8]{Ze05} (here $\downarrow$ signifies a monotonically \emph{decreasing} sequence), we will have that $\lim_{c \downarrow a} \zeta(\ell; (c,d)) = \infty$ if and only if $\ell \leq \cona$. Similarly, $\lim_{d \uparrow b} \zeta(\ell; (c,d)) = \infty$ if and only if $\ell \leq \conb$. We will be interested in the precise \emph{divergence rate} as $c \downarrow a$ or $d \uparrow b$ of these quantities. 

We will now state the key technical result on which most of our applications are based.
\begin{theorem}\label{TheoremUV}
     Let $\cona$ be as in Definition \ref{DefKappa}, and denote by $u_a(z,x)$, $v_a(z,x)$ principal resp.~nonprincipal solutions of $\tau f =zf$ at the endpoint $x = a$. Then as $x \downarrow a$ we have 
     \begin{align}\label{asymP}
         u_a(z,x) &\propto u_a(0,x) \exp{\Big\lbrace\sum_{\ell=1}^{\cona} \frac{z^\ell}{\ell} \zeta(\ell; (x,d))\Big\rbrace},  \\ \label{asymNP} v_a(z,x) &\propto v_a(0,x) \exp{\Big\lbrace-\sum_{\ell=1}^{\cona} \frac{z^\ell}{\ell} \zeta(\ell; (x,d))\Big\rbrace}.
     \end{align}
     The same result holds at the endpoint $x = b$, replacing $a$ by $b$, $\zeta(\ell; (x,d))$ by $\zeta(\ell; (c,x))$ and $\cona$ by $\conb$. 
\end{theorem}
The above result generalizes \cite[Theorem 4.2]{PS24} which roughly states that if $\cona = 0$ (resp.~$\conb = 0$), that is, the resolvent of the self-adjoint realizations of $\tau|_{(a,c)}$ (resp.~$\tau|_{(c,b)}$) is trace class, the behavior of (non)principal solutions will not depend on the spectral parameter. 

The proof of the asymptotic formula \eqref{asymNP} for the nonprincipal solution is surprisingly simple and is based on the Hadamard factorization theorem for entire functions of finite order, see Section \ref{Sect::NonPrincipal}. In Remark \ref{RemarkAltTheta} we even mention an alternative way of deriving it using, instead, a power-series approach. The analogous formula \eqref{asymP} for the principal solution uses a similar idea, but requires an additional identity for partial $\zeta$-values to hold (see Prop.~\ref{Prop:RankOne}).

Regarding further applications of Theorem \ref{TheoremUV}, formulas \eqref{asymP}, \eqref{asymNP} can be viewed as a generalization of the classic Liouville--Green (or WKB) approximations which usually requires smoothness assumptions on the coefficients, see \cite[Ch.~6]{Ol97}. A detailed comparison of those two methods is provided in Section \ref{Sect:LG}. One consequence of this comparison is the existence of \emph{approximate} integral formulas for partial $\zeta$-values in cases where the Liouville--Green approximation holds:
\begin{theorem}
    Assume that $\tau$ is in the Schr\"odinger form, that is $\tau = -\frac{d^2}{dx^2} + q$, with $x \in (a,\infty)$, and that the leading Liouville--Green approximation holds as $x \to \infty$ (see Hypothesis \ref{Hypo::LG} for the precise statement). Then for any choice of $c, c'\in (a,\infty)$, and $\ell \in \bbN_+$ we have that\\[1mm]
   $(i)$ $\lim_{x \to \infty} \int_c^x q(t)^{\frac{1}{2} - \ell} dt = \infty$ if and only if $\ell \in \lbrace 1, \dots,\coninf \rbrace$;\\[1mm]
       $(ii)$ $\lim_{x \to \infty} \Big[\frac{(2\ell-3)!!}{2^\ell (\ell-1)!}\int_{c}^x q(t)^{\frac{1}{2}-\ell} dt - \zeta(\ell; (c',x))\Big] \in \bbR$, for $\ell \in \lbrace 1, \dots,\coninf \rbrace$.
\end{theorem}
For our final application found in Section \ref{Sect:EigenConv}, we turn to the problem of the convergence rates of individual eigenvalues $\lambda_n(a,x) \to \lambda_n(a,b) = \lambda_n$ as $x \uparrow b$ under the additional assumption that $\cona = 0$. The following results show that these convergence rates for different $n$ are intimately related to each other via the partial $\zeta$-values. 
\begin{proposition}\label{PropJM}
    Assume $\cona = 0$ and $\conb \in \bbN$. Then 
    \begin{align}
        \lambda_j(a,x) - \lambda_j \propto \big(\lambda_m(a,x) - \lambda_m\big)\exp{\Big\lbrace2\sum_{\ell=1}^{\conb} \frac{\lambda_j^\ell - \lambda_m^\ell}{\ell} \zeta(\ell; (a,x))\Big\rbrace}, \quad \text{for} \ x \uparrow b. 
    \end{align}
\end{proposition}
While the above result compares the convergence rates for \emph{distinct} $n \in \lbrace j,k \rbrace$, we also prove a bound on the convergence rate which holds for all $n \in \bbN$.
\begin{corollary}
Assume $\cona = 0$ and $\conb \in \bbN$. Then for any $K>0$ and $n \in \mathbb{N}$, we have
    \begin{align}
        \lambda_n(a,x) - \lambda_n = o\Big(\exp{\Big\lbrace-K\zeta(1; (a,x))\Big\rbrace} \Big), \quad \text{for} \ x \uparrow b.
    \end{align}
\end{corollary}
The proofs of both results can be found in Section \ref{Sect:EigenConv}. We also offer numerical confirmation in the Laguerre example via the plots in Figure~\ref{Fig:Lag}.

As the partial $\zeta$-values play such an important role in all of our results, we conclude this introduction with the following proposition regarding the exact representations of these values:

\begin{proposition}\label{Prop:ExactZeta}
Assume Hypothesis \ref{Hypo:Trace} $(i)$. In addition, choose $c\in(a,b)$ such that the Friedrichs eigenvalues, $\lambda_j(c,x)$, of $\tau|_{(c,x)}$ are nonzero for all $x\in(c,b)$. Then the generating function of the partial $\zeta$-values $\zeta(\ell; (c,x)) = \sum_{j \in \bbN} \lambda_j(c,x)^{-\ell}$ is given by
\begin{align}\label{GenFct}
    -\frac{d}{dz}\ln(\varphi_c(z,x)) = \sum_{\ell = 1}^\infty \zeta(\ell; (c,x)) z^{\ell-1},
\end{align}
where $\varphi_c(z,x)$ is a solution to $\tau f = zf$ satisfying  $\varphi_c(z,c)=0$ and $\partial_x\varphi_c(z,x)|_{x = c} = 1$.
Furthermore, one has
\begin{equation}
\zeta(\ell; (c,x))=-\ell b_\ell(x),\quad \ell\in\bbN,
\end{equation}
where, with $v_c$ a nonprincipal solution at $x=c$ satisfying $W(v_c,u_c)=1$,
\begin{align}
& b_1(x)=\frac{\varphi_{c,1}(x)}{\varphi_{c,0}(x)},\quad b_j(x)=\frac{\varphi_{c,j}(x)}{\varphi_{c,0}(x)}-\sum_{k=1}^{j-1}\left(\dfrac{k}{j}\right)\frac{\varphi_{c,j-k}(x)}{\varphi_{c,0}(x)}b_k(x),\\
& \varphi_{c,0}(x) = u_c(0,x) \notag,\\
& \varphi_{c,j}(x) = \int_c^x [ u_c(0,t)v_c(0,x) - v_c(0,t)u_c(0,x)]\varphi_{c,j-1}(t)r(t) dt,\quad x\in(c,b),\ j\in\bbN. \label{phin}
\end{align}

If Hypothesis \ref{Hypo:Trace} $(ii)$ holds as well, then the choice $c=a$ is allowed wherein $\varphi_a(z,x)$ is assumed to be a principal solution at $x=a$ satisfying \eqref{PhiOne}.
\end{proposition}
\begin{proof}
The statement regarding the generating function follows by direct computation using the representation $\varphi_c(z,x) = \varphi_c(0,x) \prod_{j \in \bbN} \big(1- z/\lambda_j(c,x)\big)$, which holds whenever $c$ is chosen such that $\varphi(0,x) \not = 0$ for $x \in (c, b)$ (see the argument after Hyp.~\ref{Hypo:Trace} for details).

The second part of the proposition follows as in \cite[Thm. 2.8]{FPS25a} (see also \cite[Thm. 4.2]{FGKS21}, \cite[Thm. 3.7]{FPS25}) along with the fact that the series for $\varphi_c(z,x)$ can be written as \cite[Sect.~3]{PS24} $\varphi_c(z,x) = \sum_{j = 0}^\infty \varphi_{c,j}(x)z^j,$ $x\in(c,b),$ $z\in\bbC.$
\end{proof}
We present an alternative \emph{asymptotic} formula for $\zeta(1; (c,x))$ in Appendix \ref{appendix}.
\subsection{Notation}\label{Sect:notation}
\begin{itemize}
    \item $\sigma(T)$ and $\rho(T)$ will denote the spectrum and resolvent sets of an operator $T$, respectively. 
    \item $\kappa$, $\kappa_a$, $\kappa_b$ denote the exponents of convergence associated with the spectrum of self-adjoint realizations of $\tau$, $\tau|_{(a,c)}$, $\tau|_{(c,b)}$, respectively (none of which will depend on the particular self-adjoint realization or choice of $c\in(a,b)$\footnote{For a brief argument as to why $\kappa_a$, $\kappa_b$ do not depend on the choice of $c \in (a,b)$, see \cite[Lem.~6.3]{KST_IMRN}.}).
    \item Analogously, $\con$, $\cona$, $\conb$ denote the ranks associated with the spectrum of self-adjoint realizations of $\tau$, $\tau|_{(a,c)}$, $\tau|_{(c,b)}$, respectively.
    \item $f \sim g$ is shorthand for $\frac{f}{g} \to 1$ (in the appropriate limit); $f \propto g$ is shorthand for $\frac{f}{g} \to c \in \bbC \setminus \lbrace 0 \rbrace$.

    \item $\varphi_c(z,x)$, $\theta_c(z,x)$ will denote principal resp.~nonprincipal solutions of $\tau f = zf$ at the endpoint $c \in \lbrace a, b \rbrace$, that are entire in the spectral parameter $z$ (and depending on the context they might satisfy additional normalization conditions).
    
    \item In cases where we require a principal resp.~nonprincipal solution only for a \emph{fixed} $z \in \mathbb C$, we will use the notation $u_c(z,x)$, $v_c(z,x)$ (note that we can always choose $u_c(z,x) = \varphi_c(z,x)$ etc.).
\end{itemize}

\subsection{Outline of paper}
\begin{itemize}
    \item Section \ref{s3} contains a very brief background on singular Sturm--Liouville operators, focusing on what is necessary for our results.
    \item Section \ref{Sect:GeneralConstructions} concerns the construction of characteristic functions of minimal growth. We begin in Section \ref{Sect:Trace} with the simpler case of Sturm--Liouville expressions with either $\cona = 0$ or $\conb = 0$. For the general case, we construct nonprincipal entire solutions satisfying a certain normalization (cf.~\eqref{NormTheta}) in Section \ref{Sect::NonPrincipal} and then construct principal solutions of minimal order in Section \ref{Sect::Principal}. Finally, the construction of characteristic functions of minimal order is contained in Section \ref{Sect:Char}.
    \item Section \ref{Sect:Applications} considers two other applications of our results: connections to Liouville--Green approximation and convergence rates of eigenvalues of truncated problems.
    \item In Section \ref{Sect:Examples} our results are illustrated in the case of the harmonic oscillator and the Laguerre equation (see also Sections \ref{Sect:Airy} for Airy and \ref{Sect:PowerPotentials} for general power potentials).
    \item Appendix \ref{appendix} contains a lengthy proof of a useful integral formula for the first partial-$\zeta$ value. We also gather some observations on the principal and nonprincipal solutions which only require the weaker assumption of a purely discrete spectrum.
\end{itemize}
Multiple open problems are included throughout; see Open Problems \ref{Prob:Nonprincipal}, \ref{ProblemLG}, \ref{Prob:Conv}, and \ref{Prob:Ess}. 

%%%%%%%%%%%%%%%%%%%%%%%%%%%%%%%
%%%%%%%%%%%%%%%%%%%%%%%%%%%%%%%
\section{Background on Singular Sturm--Liouville operators} \lb{s3}
%%%%%%%%%%%%%%%%%%%%%%%%%%%%%%%
%%%%%%%%%%%%%%%%%%%%%%%%%%%%%%%

In this section we recall the standard assumptions of singular Sturm–Liouville theory (see e.g. \cite[Ch.~9]{Ze05}) that will be imposed throughout the paper.
%%%%%%%%%%%%%
\begin{hypothesis} \label{h1}
Let $(a,b) \subseteq \bbR$ and suppose that $p,q,r$ are $($Lebesgue\,$)$ measurable functions on $(a,b)$ 
such that $r,p>0$ a.e.~on $(a,b)$, $q$ is real-valued a.e.~on $(a,b)$, and $r,1/p,q\in\Ll$.
\end{hypothesis}
%%%%%%%%%%%%%
The functions introduced in Hypothesis \eqref{h1} are then used to construct the general three coefficient differential expression 
\begin{align}
    \tau=\f{1}{r(x)}\left[-\f{d}{dx}p(x)\f{d}{dx} + q(x)\right] \quad  \text{for a.e.} \ x\in(a,b) \subseteq \R.
\end{align}
A Sturm--Liouville operator is defined as the pair consisting of the differential expression $\tau$ and a suitable domain in $\Lr$ over which $\tau$ acts. The analysis of the self-adjoint extensions begin with the notion of maximal and minimal operators associated with $\tau$. The \textit{maximal operator}, $T_{max}$, and \textit{minimal operator}, $T_{min}$, in $L^2((a,b);rdx)$ associated with $\tau$ are defined, respectively, by
\begin{align}
&T_{max} f = \tau f,\ f \in \dom(T_{max})=\big\{g\in L^2((a,b);rdx) \, \big| \,g,g'\in\ACl;    
\tau g\in L^2((a,b);rdx)\big\}, \\
&T_{min} f = \tau f,\ f \in \dom(T_{min})=\big\{g\in\dom(T_{max})  \, \big| \, W(h,g)(a)=0=W(h,g)(b) \, 
\text{for all } h\in\dom(T_{max}) \big\},   \no 
\end{align}
where $W(f,g)(x) = f(x)g^{[1]}(x) - f^{[1]}(x)g(x)$ with $y^{[1]}(x)=p(x)y'(x)$ for $f,g\in\ACl$, $x \in (a,b)$,
is the Wronskian of the functions $f$ and $g$. 

According to the theory of singular Sturm--Liouville operators (see e.g. \cite[Sect. 4]{GLN20}, \cite[Ch.~13]{GNZ23}, \cite[Chs.~4, 6--8]{Ze05}, as well as \cite{CGN16}, \cite[Chs.~8, 9]{CL85}, \cite[Sects.~13.6, 13.9, 13.10]{DS88}, \cite[Ch.~XI]{Ha02}, \cite[Ch.~III]{JR76}, \cite[Ch.~V]{Na68}, \cite{NZ92}, \cite[Ch.~6]{Pe88}, \cite[Ch.~9]{Te14}, \cite[Sect.~8.3]{We80}, and \cite[Ch.~13]{We03} for a complete overview of the subject)
the self-adjoint extensions of $T_{min}$ depend on the limit point or limit circle classification of the endpoints $\{a,b\}$ of the interval. 
Since a complete analysis of the spectral $\z$-function associated with a Sturm--Liouville operator in the case in which both endpoints of the interval are limit circle (namely, the quasi-regular case) has been presented in \cite{FPS25}, we focus here on the case in which {\it at least one} endpoint is in the limit point case. In addition, we will develop our arguments under the assumption that the minimal operator $T_{min}$ (and hence all of its symmetric extensions) is bounded from below. This assumption is justified by the fact that in applications, especially in the ambit of quantum field theory, the Hamiltonian describing a physical system must be bounded from below to avoid the existence of states with arbitrarily negative energy.    
We collect the above assumptions in the following:
%%%%%%%%%%%%%%%%%%
\begin{hypothesis}\label{h2} 
The Sturm--Liouville differential expression $\tau$ satisfies Hypothesis \ref{h1}, at least one endpoint is in the limit point case, and $T_{min}$ is bounded from below, that is, there exists $\lambda_0\in\R$ such that $(u, [T_{min} - \lambda_0 I]u)_{L^2((a,b);rdx)}\geq 0$ for $u \in \dom(T_{min})$.
\end{hypothesis}
%%%%%%%%%%%%%%%%%%%%

In order to construct the self-adjoint extensions of a minimal Sturm--Liouville operator, one restricts the maximal domain by imposing suitable boundary conditions on its functions. Due to the singular nature of the endpoints of the interval, standard boundary conditions cannot, in general, be imposed on the functions of the maximal domain. For this reason, one introduces the notion of {\it generalized boundary values} which replaces the standard boundary conditions in the description of the self-adjoint extensions of a singular Sturm--Liouville operator. This is described next.

As the operator $T_{min}$ is assumed to be bounded from below, one can introduce the notion of principal and nonprincipal solutions of the equation $\tau u=\lambda u$ at the endpoints. The interested reader can find a detailed analysis of this topic in \cite[Sect. 4]{GLN20}, \cite[Ch.~13]{GNZ23}, and \cite[Chs.~4, 6--8]{Ze05}. Let $T_{min}$ be bounded from below, then for fixed $c \in (a,b)$, there exists a $\nu_0\in\bbR$ such that for all $\lambda<\nu_0$, $\tau u = \lambda u$ has $($real-valued\,$)$ nonvanishing solutions
$u_a(\lambda,\dott) \neq 0$,
$v_a(\lambda,\dott) \neq 0$ in $(a,c]$ satisfying
\begin{align}
&W(v_a (\lambda,\dott),u_a (\lambda,\dott)) = 1,
\quad u_a (\lambda,x)=\oh(v_a (\lambda,x))
\text{ as $x\downarrow a$,}  \\
&\int_a^c dx \, p(x)^{-1}u_a(\lambda,x)^{-2}=\infty, \quad \int_a^c dx \, p(x)^{-1}{v_a(\lambda,x)}^{-2}<\infty,
\end{align}
with analogous statements holding for $x=b$. See \cite{LM36} for proofs (see also \cite{Re43,Re51}, \cite[Appendix]{HW55}).
\begin{definition}\label{Def3.5}
Suppose that $T_{min}$ is bounded from below, and let 
$\lambda\in\bbR$. Then $u_a(\lambda,\dott)$ $($resp., $u_b(\lambda,\dott)$$)$ described above is called a {\it principal} $($or {\it minimal}\,$)$
solution of $\tau u=\lambda u$ at $a$ $($resp., $b$$)$. A real-valued solution 
$v_a(\lambda,\dott)$ $($resp., $v_b(\lambda,\dott)$$)$ of $\tau
u=\lambda u$ linearly independent of $u_a(\lambda,\dott)$ $($resp.,
$u_b(\lambda,\dott)$$)$ is called {\it nonprincipal} at $a$ $($resp., $b$$)$.
\end{definition}
As we will only consider problems for which self-adjoint realizations have a purely \emph{discrete} spectrum,  principal and nonprincipal solutions at both endpoints will even exist for \emph{all} parameters $\lambda \in \mathbb R$. In fact, we even allow for complex spectral parameters, as explained in the following remark.

\begin{remark}
For the sake of simplicity, we extend the notion of principal and nonprincipal solutions to the case of \emph{complex} spectral parameter $z \in \mathbb C$ via the requirement $|u_a(z,x)| = o(|v_a(z,x)|)$ as $x \downarrow a$. It will be a consequence of our construction that for problems with a purely discrete spectrum and finite exponent of convergence of eigenvalues, the prior characterization is consistent with $u_a(z,x) \in L^2((a,c); rdx)$ and $v_a(z,x) \not\in L^2((a,c); rdx)$, in case $a$ is a limit point endpoint.
\end{remark}

The introduction of principal and nonprincipal solutions allows one, in the case of $T_{min}$ bounded from below, to parameterize self-adjoint extensions via generalized boundary values at a limit circle endpoint, as detailed in the following two theorems:
\begin{theorem}[{\cite[Thm. 4.5]{GLN20}, \cite[Thm. 13.4.1]{GNZ23}}] \label{At3}
Assume that $\tau$ is in the limit circle case at $a$ and that $T_{min}$ is bounded from below. Denote by 
$u_a(\lambda_0, \dott)$ and $v_a(\lambda_0, \dott)$  principal and nonprincipal solutions of $\tau u = \lambda_0 u$ at $a$, normalized so that $W(v_a(\lambda_0,\dott), u_a(\lambda_0,\dott)) = 1.$
Then, for all $g \in \dom(T_{max})$,
\begin{align}
\wti g(a) & =  W(g,u_a(\lambda_0,\dott))(a)= \lim_{x \downarrow a} \f{g(x)}{v_a(\lambda_0,x)},\quad {\wti g}^{\, \prime}(a) = W(v_a(\lambda_0,\dott), g)(a) = \lim_{x \downarrow a} \f{g(x) - \wti g(a) v_a(\lambda_0,x)}{u_a(\lambda_0,x)}.  
\label{A6} 
\end{align}
In particular, the limits in \eqref{A6} exist. Similar results hold when the endpoint $b$ is in the limit circle case with the simple replacement $a\leftrightarrow b$. 
\end{theorem}

\begin{theorem}[{\cite[Thm. 5.5.8 \& Rem. 5.5.9]{GNZ23}, \cite[Thm. 3.11 \& Rem. 3.12]{GLN20}}]\label{extensions}
Let $\tau$ satisfy Hypothesis \ref{h2}. 
Then the following hold$:$\\[1mm]
$(i)$ If $\tau$ is in the limit point case at both endpoints, then $T_{min}=T_{max}$ is essentially self-adjoint. We denote this unique self-adjoint extension by $T_0$. \\[1mm]
$(ii)$  If $x=a$ is in the limit circle case and $x=b$ is in the limit point case, then all self-adjoint extensions $T_{\al}$ of $T_{min}$ are of the form
\begin{align}
& T_{\al} f = \tau f, \quad \al\in[0,\pi),\quad f \in \dom(T_{\al})=\big\{g\in\dom(T_{max}) \, \big| \, \cos(\al)\wti g(a)+ \sin(\al){\wti g}^{\, \prime}(a)=0\big\}.\label{A1}   
\end{align}
$(iii)$  If $x=b$ is in the limit circle case and $x=a$ is in the limit point case, then all self-adjoint extensions $T_{\b}$ of $T_{min}$ are of the form
\begin{align}
& T_{\b} f = \tau f, \quad \b\in[0,\pi),   \quad f \in \dom(T_{\b})=\big\{g\in\dom(T_{max}) \, \big| \, \cos(\b)\wti g(b)- \sin(\b){\wti g}^{\, \prime}(b)=0\big\}.\label{B1} 
\end{align}

\noindent We will denote by $T_A$ the self-adjoint extensions described in $(i)$--\,$(iii)$ and recall that the Friedrichs extension coincides with $T_0$ in all cases.
\end{theorem}

\section{Construction of principal/nonprincipal solutions and characteristic functions}\label{Sect:GeneralConstructions}

Before beginning our analysis, we once again refer the reader to Section \ref{Sect:notation} regarding specific notation used throughout this section. 

%%%%%%%%%%%%%%%%%%%%%%%%%%%%%%%
%%%%%%%%%%%%%%%%%%%%%%%%%%%%%%%
\subsection{Trace class resolvent case at one endpoint}\label{Sect:Trace}

As a primer, we start with the following simplifying assumptions, of which $(ii)$ can be found in \cite[Hyp.~2.3]{PS24}:
\begin{hypothesis}\label{Hypo:Trace}
    $(i)$ Hypothesis \ref{h3a}  holds and $x=b$ is in the limit point case.\\[1mm]
    $(ii)$ The Sturm--Liouville differential expression $\tau|_{(a,c)}$ for $c \in (a,b)$ has self-adjoint realizations with trace class resolvents, that is, $\cona=0$.
\end{hypothesis}

Conditions equivalent to the resolvents being trace class can be found in \cite[Thm. 4.2]{PS24}. 

According to \cite[Cor.~3.4]{PS24}, Hypothesis \ref{Hypo:Trace} implies that we can choose an entire generalized eigensolution  $\varphi_a(z,x)$ of $\tau f = z f$, which is principal at $x = a$ and satisfies the normalization condition
\begin{align}\label{PhiOne}
    \lim_{x \downarrow a} \frac{\varphi_a(z_1, x)}{\varphi_a(z_2, x)} = 1, \quad \text{for all} \ z_1, z_2 \in \bbC
\end{align}
(in fact the reverse implication also holds, see \cite[Thm.~4.2]{PS24}). As $\varphi_a(z,x)$ is entire in $z$ for each fixed $x$, it can be written as
\begin{align}\label{PhiFormula}
    \varphi_a(z,x) = \varphi_a(0,x)  e^{g(z)}\prod_{j \in \N} \Big(1- \frac{z}{\lambda_j(x)}\Big), \quad x \in (a,b),
\end{align}
where $\lambda_j(x)$ ($= \lambda_j(a,x)$) are the Friedrichs eigenvalues corresponding to the Sturm--Liouville differential expression $\tau|_{(a,x)}$ and $g(z)$ is some entire function with $g(0) = 0$. Note that the infinite product converges due to Hypothesis \ref{Hypo:Trace} $(ii)$. Now as $\lambda_j(x) \to \infty$ as $x \downarrow a$, it follows that $\lim_{x\downarrow a} \frac{\varphi_a(z,x)}{\varphi_a(0,x)} = e^{g(z)}$. Hence from \eqref{PhiOne} we can deduce that $g \equiv 0$.

Next, we can w.l.o.g.~assume that $0$ is not an eigenvalue of the Friedrichs realization  $T_0$ of $\tau$ on $(a,b)$ (otherwise replace $0$ by some $z_1\in\rho(T_0)$ in all formulas that follow). As both endpoints are nonoscillatory, for $x$ close enough to $b$, $\varphi_a(0,t)$ will not vanish for $t \in (x,b)$, that is, $\lambda_j(t) \not = 0$ for all $j \in \bbN$. We can thus write
\begin{align}\label{ratioChi}
    \frac{\varphi_a(z,x)}{\varphi_a(0, x)}\exp{\Big\lbrace\sum_{\ell=1}^{\conb} \sum_{j \in \N} \frac{1}{\ell}\Big(\frac{z}{\lambda_j(x)}\Big)^\ell\Big\rbrace} = \prod_{j \in \N} \underbrace{\Bigg[\Big(1- \frac{z}{\lambda_j(x)}\Big) \exp{\Big\lbrace\sum_{\ell=1}^{\conb} \frac{1}{\ell}\Big(\frac{z}{\lambda_j(x)}\Big)^\ell\Big\rbrace}\Bigg]}_{\quad \qquad E\big(\frac{z}{\lambda_j(x)}, \conb \big), \ \text{see } \eqref{c1}}.
\end{align}
The exponential on the left, which can be written as $\exp\big\lbrace\sum_{\ell=1}^{\conb} \frac{z^\ell}{\ell} \zeta(\ell;x) \big\rbrace$ with $\zeta(\ell;x) = \sum_{j \in \N} \frac{1}{\lambda_j(x)^\ell}$, is meant to regularize the product on the right. More precisely, since the sequence $\lbrace \lambda_j \rbrace_{j \in \N}$ with $\lambda_j = \lambda_j(b)$ has rank $\conb$, we can show that both sides must converge in the limit $x \uparrow b$. To see this, note that $\frac{d}{dx} \lambda_j(x) < 0$ \cite{Kong96}, and thus $\lambda_j(x) \downarrow \lambda_j$ as $x \uparrow b$. Hence, for $x$ close enough to $b$, we have that
\begin{align}
    \Big(1- \frac{z}{\lambda_j(x)}\Big) \exp{\Big\lbrace\sum_{\ell=1}^{\conb} \frac{1}{\ell}\Big(\frac{z}{\lambda_j(x)}\Big)^\ell\Big\rbrace} 
    &= 1 + O\Big(\frac{1}{\lambda_j(x)^{\conb+1}}\Big), \quad \text{locally uniformly for} \ z \in \C.
    \end{align}
As $\lambda_j^{-\conb-1}$ is summable by assumption, it follows that both sides of \eqref{ratioChi} converge for $x \uparrow b$ locally uniformly to an entire function $H(z) = \prod_{j \in \N} \big[\big(1- \frac{z}{\lambda_j}\big) \exp{\big\lbrace\sum_{\ell=1}^{\conb} \frac{1}{\ell}\big(\frac{z}{\lambda_j}\big)^\ell\big\rbrace}\Big]$. Note that this is the Hadamard representation for the characteristic function of the sequence $\lbrace \lambda_j \rbrace_{j \in \N}$, and is thus suitable for the integral representation in Theorem \ref{Thrm3.4} as summarized in the following:

\begin{proposition}\label{Prop:trace}
    Assume Hypothesis \ref{Hypo:Trace} and that $0\in\rho(T_0)$ where $T_0$ denotes the Friedrichs realization of $\tau$. Let $\varphi_a(z,x)$ be an entire solution of $\tau f = zf$ which is principal at $x = a$ and satisfies the normalization \eqref{PhiLimit}. Then the Hadamard characteristic function\footnote{We refer to \cite[Thm. 2.6, Rem. 2.7 $(ii)$, Cor. 2.10, Rem. 2.11 $(ii)$]{FPS25a} for the importance of the Hadamard form of the characteristic function in the context of spectral $\zeta$-functions.} for $\sigma(T_0)$ is given via the formula
    \begin{align}\label{HFormula}
        H_0(z) = \lim_{x\uparrow b} \frac{\varphi_a(z,x)}{\varphi_a(0,x)} \exp\Big\lbrace\sum_{\ell=1}^{\conb} \frac{z^\ell}{\ell} \zeta(\ell;x) \Big\rbrace,\quad z\in\bbC. 
    \end{align}
    Note that $\varphi_a(0,x)$ in the denominator can be replaced with any other (nonprincipal at $x = b$) solution $v_b(0,x)$, in which case $H_0(z)$ would change only by a multiplicative factor.
    
    A similar statement can be made in case Hypothesis \ref{Hypo:Trace} holds at the endpoint $b$ instead, with the roles of $a$ and $b$ in \eqref{HFormula} reversed.
\end{proposition}
This result demonstrates that from the point of view of characteristic functions the normalization \eqref{PhiOne} studied in \cite{PS24} is very natural. It is also apparent that the role of the \emph{partial} $\zeta$-values $\zeta(\ell; x)$, in particular their divergence as $x \uparrow b$, is crucial. 

\begin{remark}\label{Rem:3.3}
When $\tau$ is in the limit circle case at $x = a$, we can consider instead the self-adjoint realization $T_\alpha$ defined in Theorem \ref{extensions} $(ii)$. The Hadamard characteristic function $H_\alpha$ is then given by
\begin{align}\label{Eq:gentracechar}
    H_\alpha(z) = \lim_{x\uparrow b} \frac{\psi_a(z,x)}{\psi_a(0,x)} \exp\Big\lbrace\sum_{\ell=1}^{\conb} \frac{z^\ell}{\ell} \zeta_\alpha(\ell;x) \Big\rbrace, \quad z\in\bbC,
\end{align}
with $\psi_a(z,x) = \cos(\alpha) \varphi_a(z,x) - \sin(\alpha) \theta_a(z,x)$ and $\zeta_\alpha(\ell; x) = \sum_{j \in \bbN} \lambda_{\alpha, j}(x)$. Here $\lbrace \lambda_{\alpha, j}(x) \rbrace_{j \in \bbN}$ is the spectrum of the self-adjoint realization of $\tau|_{(a,x)}$ with $\alpha$-boundary condition at $a$ and Dirichlet boundary condition at $x$.
\end{remark}

\subsubsection{The Airy equation}\label{Sect:Airy}
To illustrate the above results, we now consider the Airy equation $-y''(x)+ x y(x)=z y(x)$ for $x\in(0,\infty)$. Following \cite[Sect. 5.4]{FGKS21}, we can write the solution satisfying $\alpha$-boundary conditions at $x=0$ as $\psi_0(z,x) = \cos(\alpha) \varphi_0(z,x) - \sin(\alpha) \theta_0(z,x)$ with 
\begin{align}
\begin{split}\label{5.76}
\varphi_0(z,x)&=\pi [\Ai(-z)\Bi(x-z)-\Bi(-z)\Ai(x-z)],\\ 
\theta_0(z,x)&=-\pi [\Ai'(-z)\Bi(x-z)-\Bi'(-z)\Ai(x-z)],
\quad z\in\C,
\end{split}
\end{align}
where $\Ai(\dott)$ and $\Bi(\dott)$ represent the Airy functions of the first and second kind, respectively (cf. \cite[Sect. 9]{DLMF}). Note that $\varphi_0$ satisfies \eqref{PhiOne} and the solutions $\theta_0,\varphi_0$ have unit Wronskian.

As the resolvents for this problem are Hilbert--Schmidt but not trace class (the eigenvalue growth is proportional to $n^{2/3}$ so that $\kappa_\infty=3/2$), we conclude that only the first partial $\zeta$-value in \eqref{Eq:gentracechar} needs to be computed.
Furthermore, from the $x\to \infty$ asymptotic behavior of the Airy functions \cite[Sect. 9.7]{DLMF}, it is apparent that only the terms with $\Bi(x-z)$ will contribute to the limit in \eqref{Eq:gentracechar}, that is,
\begin{equation}
H_\alpha(z)=\frac{\cos(\alpha)\Ai(-z)+\sin(\alpha)\Ai'(-z)}{\cos(\alpha)\Ai(0)+\sin(\alpha)\Ai'(0)}\lim_{x\to\infty}\frac{\Bi(x-z)}{\Bi(x)}e^{z\zeta_\alpha(1;x)},\quad z\in\bbC,
\end{equation}
where we are assuming that the denominator is nonzero, that is, $\alpha\neq \tan^{-1}\big(-\Ai(0)/\Ai'(0)\big).$ Otherwise, the denominator can simply be removed throughout the following.

By Proposition \ref{Prop:ExactZeta}, the partial $\zeta$-value can be written explicitly as
\begin{align}
\zeta_\alpha(1;x)&=-\frac{\frac{d}{dz}\psi_0(z,x)\big|_{z=0}}{\psi_0(0,x)} \underset{x\to\infty}{=}x^{1/2}+\frac{\cos(\alpha)\Ai'(0)+\sin(\alpha)\Ai''(0)}{\cos(\alpha)\Ai(0)+\sin(\alpha)\Ai'(0)}+o(1),
\end{align}
where the $x^{1/2}$ term comes from the ratio $\Bi'(x)/\Bi(x)$. From this and \cite[Eq. 9.7.8]{DLMF} the Hadamard characteristic function \eqref{Eq:gentracechar} can be written as
\begin{equation}\label{Eq:AiryChar}
H_\alpha(z)=\frac{\cos(\alpha)\Ai(-z)+\sin(\alpha)\Ai'(-z)}{\cos(\alpha)\Ai(0)+\sin(\alpha)\Ai'(0)}\exp\Big\{z\frac{\cos(\alpha)\Ai'(0)+\sin(\alpha)\Ai''(0)}{\cos(\alpha)\Ai(0)+\sin(\alpha)\Ai'(0)}\Big\},\quad \a\in[0,\pi),\ z\in\bbC.
\end{equation}
This is an entire function of order $3/2$ \cite[Sect. 9.7]{DLMF}, and hence has minimal growth as expected.
The case $\a=0$ (i.e., Dirichlet at $x=0$) corresponds to the Airy $\zeta$-function recently studied in \cite{FPS25a} (see also \cite{Cr96,voros,voros23}).
Notice that the denominator and exponential term in the formula for $H_\alpha$ in \eqref{Eq:AiryChar} are there to guarantee that this yields the Hadamard characteristic function, so can safely be discarded if one only wants a characteristic function of minimal order.

A problem with two limit point endpoints can be constructed from the Airy problem by considering the potential $|x|$ on $\bbR$.
A straightforward way of obtaining a characteristic function for this problem relies on the fact that its spectrum is simply the union of the spectrum of the half-line problem with Dirichlet boundary conditions and the one with Neumann boundary conditions at $x=0$. This is due to the symmetry of the potential about the point $x=0$ (see e.g. the proof of Thm.~4.3 in \cite{FS25a}). Therefore, a characteristic function for the problem with potential $|x|$ on the real-line can be written as
\begin{equation}
H(z)=H_0(z)H_{\pi/2}(z)=\frac{\Ai(-z)\Ai'(-z)}{\Ai(0)\Ai'(0)}\exp\Big\lbrace z\Big(\frac{\Ai'(0)}{\Ai(0)}+\frac{\Ai''(0)}{\Ai'(0)}\Big)\Big\rbrace,\quad z\in\bbC.
\end{equation}
We will alternatively be able to construct this characteristic function without the use of the symmetry properties by utilizing the results shown the next few sections.

\subsection{Construction of nonprincipal entire solutions}\label{Sect::NonPrincipal}
The difficulty in obtaining an expression for a characteristic function of minimal order increases significantly if we drop Hypothesis \ref{Hypo:Trace} $(ii)$. For one, the partial $\zeta$-values $\zeta(\ell; x)$ no longer converge; we will replace them in the following with $\zeta(\ell; (c,x)) = \sum_{j \in \bbN} \lambda_j(c,x)^{-\ell}$, where $\lambda_j(c,x)$ are the Dirichlet eigenvalues for $\tau|_{(c,x)}$, $c \in (a,b)$. Secondly, the representation \eqref{PhiFormula} no longer converges. 

Our strategy will be to construct an entire principal solution $\varphi_a(z,x)$ at $x=a$ of minimal order $\kappa_a$  and then show that
\begin{align}\label{FFormula}
    F_0(z) = \lim_{x \uparrow b} \frac{\varphi_a(z,x)}{v_b(0,x)} \exp\Big\lbrace\sum_{\ell=1}^{\conb} \frac{z^\ell}{\ell} \zeta(\ell;(c,x)) \Big\rbrace,
\end{align}
defines a characteristic function of minimal order $\kappa$, where $v_b(0,x)$ is any nonprincipal solution at $x = b$ (though $F_0$ will, in general, no longer be the Hadamard characteristic function). As a first step, we will prove that the regularization via the exponential $\exp\Big\lbrace\sum_{\ell=1}^{\conb} \frac{z^\ell}{\ell} \zeta(\ell;(c,x)) \Big\rbrace$ leads to a finite limit in \eqref{FFormula}, irrespective of the choice of $c \in (a,b)$. To this end, we need to study the behavior of nonprincipal solutions near $x = b$ for different eigenparameters.
\\
\\
Let $\varphi_c(z,x)$ be an entire solution of $\tau f = zf$ satisfying $\varphi_c(z,c) = 0$ and $\partial_x \varphi_c(z,x)|_{x =c} = 1$. Just as in the previous section, it follows that $G_c(z) = \lim_{x \uparrow b} \frac{\varphi_c(z,x)}{\varphi_c(0,x)} \exp\Big\lbrace\sum_{\ell=1}^{\conb} \frac{z^\ell}{\ell} \zeta(\ell;(c,x)) \Big\rbrace$ exists and defines the Hadamard characteristic function for the Friedrichs realization $T_{(c,b)}$ of the truncated problem $\tau|_{(c,b)}$. Moreover, $\varphi_c(z,x)$ will be nonprincipal at $x = b$ for $z \not \in \sigma(T_{(c,b)})$. 

To construct an appropriate \emph{entire} solution which is nonprincipal at $x = b$ for \emph{all} $z \in \bbC$ we proceed as follows: Choose any $\lambda$ disjoint from the zeros of $G_c(z)$. It follows that, for any neighbourhood $U$ of $\lambda$ disjoint from $\sigma(T_{(c,b)})$, the function $\theta_U(z,x) = \varphi_c(z,x)/G_c(z)$ is well-defined and satisfies
\begin{align}\label{thetaU}
    \lim_{x \uparrow b} \frac{\theta_U(z,x)}{\varphi_c(0,x)} \exp{\Big\lbrace\sum_{\ell=1}^{\conb} \frac{z^\ell}{\ell} \zeta(\ell; (c,x))\Big\rbrace} = 1, \quad z \in U.
\end{align}
Note the restriction $\lambda \not = \lambda_j$ in the construction of $\theta_U$. However, this issue can be easily resolved by perturbing $c \to c'$, so that $\lambda_j$ is no longer an eigenvalue
of $T_{(c', b)}$ and choosing a $\varphi_{c'}(z,x)$ with $\lim_{x \uparrow b} \frac{\varphi_{c'}(0,x)}{\varphi_c(0,x)} = 1$ and defining $G_{c'}(z)$ as before. This in particular shows that $\zeta(\ell; (c', x)) = \zeta(\ell; (c, x)) + O(1)$ as $x \uparrow b$.

It follows that around \emph{any} $z \in \C$, there exists a corresponding holomorphic solution $\theta_U(z,x)$ such that \eqref{thetaU} is satisfied. Now take an arbitrary pair of entire solutions $\widetilde \varphi_b(z,x)$, $\widetilde \theta_b(z,x)$ which are principal resp.~nonprincipal at the endpoint $x = b$ (their existence is guaranteed by \cite[Lem.~2.2, 2.4]{KST_IMRN}). Then
\begin{align}
    \widetilde \theta_b(z,x) = h_U(z) \theta_U(z,x) + g_U(z) \widetilde \varphi_b(z,x), \quad h_U(z) \not = 0, \ z \in U.
\end{align}
It follows that 
\begin{align}
    \lim_{x \uparrow b} \frac{\widetilde \theta_b(z,x)}{\varphi_c(0,x)}\exp{\Big\lbrace\sum_{\ell=1}^{\conb} \frac{z^\ell}{\ell} \zeta(\ell; (c,x))\Big\rbrace} = h_U(z).
\end{align}
Thus the above limit exists for all $z \in \C$ and defines a nonvanishing entire function, call it $h(z)$. In particular $\theta_b(z,x) = \widetilde \theta_b(z,x)/h(z)$ will satisfy the normalization
\begin{align}\label{NormTheta}
    \lim_{x\uparrow b} \frac{\theta_b(z,x)}{\theta_b(0,x)} \exp{\Big\lbrace\sum_{\ell=1}^{\conb} \frac{z^\ell}{\ell} \zeta(\ell; (c,x))\Big\rbrace} = 1.
\end{align}
This can be viewed as the generalization of the normalization found in \cite[Cor.~4.5]{PS24}. Alternatively, we can write $\lim_{x \uparrow b} \frac{\theta_b(z_1, x)}{\theta_b(z_2, x)} \exp{\Big\lbrace\sum_{\ell=1}^{\conb} \frac{z_1^\ell -z_2^\ell}{\ell} \zeta(\ell; (c,x))\Big\rbrace} = 1$.

We will see in Theorem \ref{TheoremChar} that the normalization \eqref{NormTheta} allows us to write formulas for the characteristic functions of minimal order that do not include partial $\zeta$-values $\zeta(\ell; (c,x))$. This is useful in situations in which the partial $\zeta$-values are difficult to compute. On the other hand, approximate trace formulas for partial $\zeta$-values in case of Schr\"odinger operators satisfying Hypothesis \ref{Hypo::LG} can be found in Theorem \ref{TheoremLG}. Alternatively, one can use the generating function provided in Proposition \ref{Prop:ExactZeta}.

\begin{remark}\label{RemarkAltTheta}
Note that Proposition \ref{Prop:ExactZeta} offers an alternative way of computing the asymptotics of nonprincipal solutions. We obtain 
\begin{align}
\begin{split}
    \varphi_c(z,x) &= \varphi_c(0,x)\exp\Big\lbrace \ln\, \varphi_c(z,x) - \ln \, \varphi_c(0,x) \Big\rbrace
    \\
    &= \varphi_c(0,x)\exp\Big\lbrace - \int_0^z \sum_{\ell=1}^\infty \zeta(\ell; (c,x)) s^{\ell-1}\, ds \Big\rbrace \qquad (\text{for } \, |z| \, \text{ small enough})
    \\
    &= \varphi_c(0,x)\exp\Big\lbrace -  \sum_{\ell=1}^{\conb} \frac{\zeta(\ell; (c,x))}{\ell} z^\ell + \text{ convergent terms in  } x \Big\rbrace,
    \end{split}
    \end{align}
where we have split the divergent terms (as $x \uparrow b$) of the sum from the rest which is convergent. Ignoring the convergent terms, we arrive at the correct asymptotics \eqref{asymNP} (at the endpoint $b$ instead of $a$, but the choice does not matter). The rest of the analysis follows a similar logic as above; we leave the details to the reader.  
\end{remark}
\\
\\
Before we proceed with the construction of principal solutions of minimal order, let us note a few results that allow us to determine Schatten $p$-class membership of resolvents of self-adjoint realizations solely in terms of the behavior of nonprincipal solutions for different eigenparameters. 

\begin{corollary}\label{CorTheta}
  Assume that the self-adjoint realizations of $\tau|_{(c,b)}$ have resolvents in the Schatten $p$-class with $p = \conb+1$, but not $p = \conb$ (i.e, $\lim_{x \uparrow b} \zeta(\conb+1; (c,x)) \in \mathbb R$ exists but $\lim_{x \uparrow b} \zeta(\conb; (c,x)) = \infty$). Let $v_b(z_j, x), v_b(w_j, x)$, $1\leq j \leq n$ denote (any) nonprincipal solutions at $x = b$. Then the limit 
  \begin{align}\label{vRatio}
      \lim_{x \uparrow b} \frac{v_b(z_1, x) \cdots  v_b(z_n,x)}{v_b(w_1, x) \cdots v_b(w_n ,x)}
  \end{align}
  exists and is nonzero if and only if
  $\sum_{j = 1}^n z_j^\ell = \sum_{j =1}^n w_j^\ell$ for $\ell = 1, \dots,\conb$. If $\theta_b(z,x)$ is normalized according to \eqref{NormTheta}, then we even have
  \begin{align} \label{NormTheta2}
      \lim_{x \uparrow b} \frac{\theta_b(z_1, x) \cdots  \theta_b(z_n,x)}{\theta_b(w_1, x) \cdots \theta_b(w_n ,x)} = 1.
  \end{align}
  In particular, this equality does not depend on the choice of $c \in (a,b)$ in \eqref{NormTheta}.
\end{corollary}
\begin{proof}
    The equality in \eqref{NormTheta2} is obtained from \eqref{NormTheta} through a direct computation. As nonprincipal solutions for a given spectral parameter $z$ are unique up to nonzero multiples, the existence of the limit \eqref{vRatio} also follows. That the condition $\sum_{j =1}^n z_j^\ell = \sum_{j=1}^n w_j^\ell$ for $\ell = 1, \dots ,\conb$ is indeed \emph{necessary} for the convergence of \eqref{vRatio}, is a consequence of the fact that $\lim_{x\uparrow b} \zeta(\ell, (c,x)) = \infty$ for $\ell = 1, \dots, \conb$, with $\zeta(\ell; (c,x))$ diverging faster than $\zeta(\ell +1; (c,x))$ (which can even converge for $\ell = \conb$), implying that an \emph{exact} cancellation in the exponential is necessary for the limit in \eqref{vRatio} to exist (cf.~the proof of Cor.~\ref{CorTildeTheta} below).   
\end{proof}
The above corollary gives us a characterization of the Schatten $p$-class resolvent condition for an integer $p$ in terms of the behavior of nonprincipal solutions, and thus generalizes \cite[Thm.~4.2(iv)]{PS24}. Note that the condition $\sum_{j = 1}^n z_j^\ell = \sum_{j =1}^n w_j^\ell$ for $\ell = 1, \dots,\conb$ is equivalent to $\sum_{j = 1}^n g_{\conb}(z_j) = \sum_{j = 1}^n g_{\conb}(w_j)$ for all polynomials $g_{\conb}$ up to degree $\conb$. 

\begin{corollary}\label{CorTildeTheta}
    In the setting of Corollary \ref{CorTheta}, any other entire nonprincipal solution $\widetilde \theta_b(z,x)$ satisfying \eqref{NormTheta2} must be of the form
\begin{align}\label{TildeTheta}
    \widetilde \theta_b(z,x) = e^{g_{\conb}(z)} \theta_b(z,x) +  \varphi_b(z,x),
\end{align}
where $\theta_b(z,x)$ satisfies \eqref{NormTheta}, $\varphi_b(z,x)$ is an arbitrary entire principal solution at $x = b$, and $g_{\conb}(z)$ is a polynomial of degree at most $\conb$.
\end{corollary}
\begin{proof}
    It is clear that any entire nonprincipal solution at $x = b$ must have the form $\widetilde \theta_b(z,x) = e^{f(z)} \theta_b(z,x) +  \varphi_b(z,x)$ for some entire function $f$. Moreover, for this choice of $\widetilde \theta_b$, the limit in \eqref{NormTheta2} will be equal to $\exp\big\lbrace\sum_{j = 1}^n f(z_j) - \sum_{j = 1}^n f(w_j)\big\rbrace$, which, in turn, will be equal to $1$ whenever $f$ is a polynomial of degree at most $\conb$. If $f$ is not of this form, we can write it as $f(z) = f_{\conb}(z) + f_mz^m + O(z^{m+1})$ with $f_{\conb}$ a polynomial of degree at most $\conb$ and $m \geq \conb+1$ with $f_m \not = 0$. Choosing $z_j, w_j$ as above and letting $\varepsilon$ be small, we see that necessarily $\sum_{j=1}^n [f(\varepsilon z_j) - f(\varepsilon w_j)] = \varepsilon^m f_m\sum_{j=1} [z_j^m - w_j^m] + O(\varepsilon^{m+1}) = 0$. By letting $\varepsilon \to 0$, it follows that $\sum_{j=1}^n z_j^m = \sum_{j=1}^n w_j^m$. Now choosing $z_j = \exp(2\pi j/m)$ for $j = 1, \dots, m$ and $w_j = 0$ we get $\sum_{j=1}^m  z_j^\ell = \sum_{j=1}^m w_j^\ell$ for $\ell = 1, \dots, m-1$ (in particular up to $\conb$), but $\sum_{j=1}^m z_j^m = m \not= 0 = \sum_{j=1}^m w_j^m$. Thus it follows that $f(z)$ must be a polynomial of degree at most $\conb$. 
\end{proof}

\begin{remark}\label{remarkTheta}
    Note that the condition \eqref{NormTheta} does not guarantee that $\theta_b(z,x)$ will be of minimal (or even finite) order. This is apparent from \eqref{TildeTheta}, as we can add an \emph{arbitrary} entire solution $\varphi_b(z,x)$ (hence not necessarily of minimal order) which is principal at $x = b$. This is in stark contrast to the \emph{principal} solution that will be analyzed in the subsequent section, where the analogous normalization \eqref{PhiLimit} is in fact equivalent to the order being minimal. Since our analysis is not affected by the nonprincipal solution not being of minimal order, we leave this as an open problem.
\end{remark}

\begin{problem}\label{Prob:Nonprincipal}
Is there a natural construction of a nonprincipal solution that guarantees it is of minimal order in direct parallel to our construction of the principal solution in Section \ref{Sect::Principal}?
\end{problem}
The above problem is important with regards to local Borg--Marchenko uniqueness results as demonstrated in the proof of \cite[Thm. 7.2]{KST_IMRN} (see also Rmk.~\ref{Remark:HN} below). While constructing (or even proving the existence) of nonprincipal solutions of minimal order in case of Schatten $p$-class resolvents appears highly nontrivial, we can show that every \emph{naturally normalized system} of principal and nonprincipal solutions (see \cite[Def.~7.1]{PS24}) in case of a finite regularization index is of minimal order $\kappa = \frac{1}{2}$. The notions of the regularization index and a naturally normalized system are recalled in the proof below.
\begin{proposition}\label{Prop:Min}
    Assume self-adjoint realizations of $\tau|_{(a,c)}$ have trace-class resolvents, $\tau$ has finite regularization index at $x=a$, and let $\varphi_a$, $\theta_a$ be a naturally normalized system of principal, nonprincipal solutions at $x = a$ (see \cite[Def.~7.1]{PS24}). Then $\theta_a(z,x)$ is of minimal order $\kappa = \frac{1}{2}$ for all $x \in (a,b)$. 
\end{proposition}
\begin{proof}
According to \cite[Sect.~7]{PS24}, being a naturally normalized entire system is equivalent to the following condition: writing $\varphi_a(z,x) = \sum_{j = 0}^\infty \varphi_{a,j}(x)(z-\lambda)^j$, $\theta_a(z,x) = \sum_{j = 0}^\infty \theta_{a,j}(x)(z-\lambda)^j$, we have
\begin{align}
    \lim_{x\downarrow a}\frac{\varphi_{a,j+1}(x)}{\varphi_{a,j}(x)} = 0, \qquad \lim_{x\downarrow a}\frac{\theta_{a,j+1}(x)}{\theta_{a,j}(x)} = 0, \qquad \text{for all } \ j \in \mathbb N_0.
\end{align}
 The regularization index (at $x = a$) is then defined via 
 \begin{align}
     \ell_a = \sup\Big\lbrace n \  \colon \lim_{x\downarrow a} \cfrac{\varphi_{a,0}(x)}{\theta_{a,n}(x)} = 0\Big\rbrace,
 \end{align}
 which is assumed here to be finite, that is, $\ell_a \in \mathbb N_0$. 
 
 Now assume that $\lambda$ is below the lowest Friedrichs eigenvalue, in which case $\varphi_a(\lambda,x) = \varphi_{a,0}(x)$ is nonvanishing and nonprincipal at $x = b$. We can assume that $\theta_{a,0}(x) = \theta_a(\lambda, x) = \varphi_a(\lambda, x) \int_x^b \frac{1}{p(t)\varphi_a^2(\lambda, t)} dt$ (any other choice would only differ by a constant multiple of $\varphi_a(\lambda, x)$), in which case $\theta_{a,0}(x)$ is nonvanishing as well. It now follows from \cite[Eq.~(3.5), (7.1)]{PS24} that for a naturally normalized system we have
 \begin{align}\label{recPhi}
     \varphi_{a,j}(x) &= \int_a^x [\varphi_a(\lambda,t)\theta_a(\lambda, x) - \theta_a(\lambda, t)\varphi_{a}(\lambda, x)]\varphi_{a, j-1}(t) r(t) dt, \qquad x \in (a,b), \ j \geq 1, 
 \\\label{recTheta}
     \theta_{a,j}(x) &= \int_a^x [\varphi_a(\lambda,t)\theta_a(\lambda, x) - \theta_a(\lambda, t)\varphi_{a}(\lambda, x)]\theta_{a, j-1}(t) r(t) dt, \qquad x \in (a,b), \ j \geq \ell_a + 1, 
 \end{align}
 and from \cite[Thm.~7.6]{PS24}
 \begin{align}
     \lim_{x\downarrow a} \frac{\varphi_{a, j}(x)}{\theta_{a, j + \ell_a}(x)} =0, \qquad \lim_{x\downarrow a} \frac{\theta_{a, j + \ell_a+1}(x)}{\varphi_{a, j}(x)} =0, \qquad \text{for all } \ j \in \mathbb N_0.
 \end{align}
 Upon choosing $j = 0$ above, we see that for any $c \in (a,b)$ there is a $C$ such that $|\varphi_0(x)| \leq C |\theta_{\ell_a}(x)|$  and  $|\theta_{\ell_a + 1}(x)| \leq C|\varphi_0(x)|$ for $x \in (a,c)$. Now due to the recursions \eqref{recPhi}, \eqref{recTheta} and the fact that $\varphi_a(\lambda,t)\theta_a(\lambda, x) - \theta_a(\lambda, t)\varphi_{a}(\lambda, x)$ has a fixed sign for $a < t < x < b$ (see \cite[Thm. 2.2(iii)]{NZ92})
    it follows by induction that 
    $|\varphi_j(x)| \leq C|\theta_{\ell_a+j}(x)| $ and $ |\theta_{\ell_a + 1+j}(x)| \leq C|\varphi_j(x)|,$ for $ x \in (a,c)$
    (it is important here that $c$ does not depend on $j$). Using the formula for the order of an entire function in terms of its Taylor coefficients (see \cite[Ch.~1.3, Thm.~2]{Levin1996}), it follows that $\varphi_a(z,x)$ and $\theta_a(z,x)$ have the same order $\kappa$ for $x \in (a,c)$. As $c$ was arbitrary, the claim remains true for $x \in (a,b)$. As a naturally normalized principal solution $\varphi_a$ is of minimal order by Corollary \ref{CorPhiTilde} (it satisfies $\lim_{x\downarrow a} \frac{\varphi_a(z,x)}{\varphi_a(0,x)} = 1$, see \cite[Def.~7.1]{PS24}), it follows that $\theta_a(z,x)$ is also of minimal order $\kappa$ for $x \in (a,b)$. Finally, $\kappa = \frac{1}{2}$ follows from \cite[p.~33]{WW14} as explained in \cite[Prop. 9.5]{PS24}.
\end{proof}
\begin{remark} \label{Remark:HN} Open Problem \ref{Prob:Nonprincipal} can also be motivated by the observation that in the trace class resolvent case of Schr\"odinger operators with finite regularization index, the associated singular $m$-functions are in the class of generalized Herglotz--Nevanlinna functions provided the solutions are naturally normalized as above -- see \cite[Sect.~10]{PS24}. We believe that in the general Schatten $p$-class resolvent case the underlying singular $m$-functions might belong to some natural extension of the generalized Herglotz--Nevanlinna class, provided \emph{both} solutions are of minimal order (see also \cite[Open Prob. 10.6]{PS24}). 

As previously mentioned, the construction of such an entire fundamental system of solutions with both solutions being of minimal growth order was considered in \cite[Sect. 6]{KST_IMRN}, albeit under additional assumptions on eigenvalue growth and their interlacing properties. The constructed system was then used to show a Borg--Marchenko type uniqueness result in \cite[Thm. 7.2]{KST_IMRN}. However, the assumptions made there do not hold for the problems we consider here. In fact, Proposition \ref{Prop:Min} can be seen as removing Hypothesis 6.1 $(ii)$ from \cite[Lem. 6.10]{KST_IMRN}. Moreover, one can see from the examples below that the typical expected interlacing of zeros of the entire principal and nonprincipal solutions need not hold if \eqref{NormTheta} holds whenever the endpoint under consideration is limit point.

As such, not only does the above problem have multiple applications, but it also seems to be challenging in the generality considered here.
\end{remark}

We finish this section with an illustrative example.

\subsubsection{Power potentials $($including harmonic oscillator and Airy$)$}\label{Sect:PowerPotentials} We begin by considering the harmonic oscillator $\tau y = -y''(x)+ x^2 y(x)=z y(x),$ for $x\in\bbR$. The special solution of $\tau y = -y$ which is nonprincipal at both endpoints is given by $v_\infty(-1, x) = e^{\frac{x^2}{2}}$ (we will focus on the endpoint $\infty$). A corresponding principal solution is readily given by $u_\infty(-1, x) = e^{\frac{x^2}{2}}\int_x^\infty e^{-t^2}dt = \frac{e^{-\frac{x^2}{2}}}{2x}(1 + o(1))$ for $x \to \infty$. From Weyl's law it follows that the eigenvalues grow linearly, implying that $\kappa_\infty = \coninf = 1$, that is, we only require $\zeta(1; (c,x))$. We can estimate the partial $\zeta$-value $\zeta(1; (c,x))$ via the formula 
\begin{align}\label{Zeta1Harm}
    \zeta(1; (c,x)) = \int_{c}^x u_\infty(\lambda, t) v_\infty(\lambda, t) r(t) dt + \underbrace{O(\zeta(2; (c,x)))}_{=O(1), \ \text{as} \, \coninf = 1} = \frac{\ln \, x}{2} + O(1),
\end{align}
which is shown in Lemma \ref{LemmaInt1}. It now follows readily from \eqref{NormTheta} that
\begin{align}\label{HarmV}
    v_\infty(z, x) \propto v_\infty(-1, x)\exp\Big\lbrace-(z+1)\frac{\ln \, x}{2} \Big\rbrace = x^{-\frac{z+1}{2}} e^{\frac{x^2}{2}},
\end{align}
which is (up to a multiplicative prefactor) exactly the leading large-$x$ asymptotics of $V(-z/2,\sqrt{2}x)$ which is one of the standard parabolic cylinder functions \cite[Sect. 12.9]{DLMF} and satisfies $\tau V(-z/2,\sqrt{2}x) = z V(-z/2,\sqrt{2}x)$. Moreover, using \eqref{asymP} (which we will show in the next section), it follows that
\begin{align}\label{HarmU}
    u_\infty(z, x) \propto u_\infty(-1, x)\exp\Big\lbrace(z+1)\frac{\ln \, x}{2} \Big\rbrace = \frac{x^{\frac{z-1}{2}}}{2} e^{-\frac{x^2}{2}},
\end{align}
which recovers the leading asymptotics of the other parabolic cylinder function $U(-z/2,\sqrt{2}x)$ which converges to $0$ at infinity \cite[Sect. 12.9]{DLMF}.

We now turn to the general power potential equation $-y''(x)+x^d y(x)=z y(x),$ $x\in(0,\infty)$, $d>0$ (note the restriction to the half-line here). Closed form solutions for general $z\in\bbC$ do not exist except in a few special cases. However, linearly independent solutions to $-y''(x)+x^d y(x)=0$ (with unit Wronskian) are given by
\begin{align}
\begin{split}
u_{\infty,d}(0,x)&=\sqrt{x}K_{1/(2+d)}\Big(\frac{2x^{(2+d)/2}}{2+d}\Big),\;\:\textrm{and}\;\;
v_{\infty,d}(0,x)=\frac{2\sqrt{x}}{2+d}I_{1/(2+d)}\Big(\frac{2x^{(2+d)/2}}{2+d}\Big),\quad d>0,
\end{split}
\end{align}
where $K_\nu,I_\nu$ are modified Bessel functions \cite[Sect. 10]{DLMF}. Note that, by \cite[Sect. 10.4]{DLMF}, only the first function is $L^2$ near infinity. 

Now, recalling that the leading order eigenvalue growth is $\lambda_n \propto n^{2d/(d+2)}$ (see \cite[Eq. (7.1.7)]{Ti62} for explicit constant), one sees that self-adjoint realizations have trace class resolvent whenever $d>2$, while the resolvents are Hilbert--Schmidt whenever $2/3<d\leq 2$. Thus we now assume that $0<d\leq 2$.

Utilizing \cite[Eqs. 10.4.1, 10.4.2]{DLMF} together with Lemma \ref{LemmaInt1} (more precisely \eqref{Lambda0Formula}) one finds as above
\begin{equation}\label{eq:zeta1}
\zeta_{d}(1;(c,x))=\begin{cases}
\frac{1}{2-d}x^{\frac{2-d}{2}}+O(1), & 0<d<2,\\
\frac{1}{2}\ln\, x +O(1), & d=2.
\end{cases}
\end{equation}

As with the special case of the Harmonic oscillator, it now follows from \eqref{NormTheta}, \eqref{asymP}, and \cite[Eqs. 10.4.1, 10.4.2]{DLMF} that
\begin{align}
    v_{\infty,d}(z, x) &\propto
    \begin{cases}
    x^{-\frac{d}{4}}\exp\big(\tfrac{2}{2+d}x^{\frac{2+d}{2}}-\frac{z}{2-d}x^{\frac{2-d}{2}}\big), & 2/3<d<2,\\[2mm]
    x^{-\frac{z+1}{2}}\exp\big(\tfrac{1}{2}x^{2}\big), & d=2,
    \end{cases}\label{eq:vinf}\\
    u_{\infty,d}(z, x) &\propto
    \begin{cases}
    x^{-\frac{d}{4}}\exp\big(-\tfrac{2}{2+d}x^{\frac{2+d}{2}}+\frac{z}{2-d}x^{\frac{2-d}{2}}\big), & 2/3<d<2,\\[2mm]
    x^{\frac{z-1}{2}}\exp\big(-\tfrac{1}{2}x^{2}\big), & d=2.
    \end{cases} \label{eq:uinf}
\end{align}
Choosing $d=2$ yields the previous results as expected. Moreover, choosing $d=1$ corresponds to the Airy differential equation (see also Sect. \ref{Sect:Airy}) and one readily verifies via \cite[Sect. 9.7(ii)]{DLMF} that the leading asymptotics found here agree with those of the Airy functions $\Bi(\dott)$ and $\Ai(\dott)$, respectively.

Finally, we point out that whenever $0<d\leq 2/3$, the above solutions will now be proportional to \eqref{eq:vinf}, \eqref{eq:uinf}  times
\begin{equation}
\exp{\bigg\lbrace\mp\sum_{\ell=2}^{\big\lfloor\tfrac{d+2}{2d} \big\rfloor} \frac{z^\ell}{\ell} \zeta_{d}(\ell; (c,x))\bigg\rbrace},\quad 0<d\leq 2/3,\ 
\end{equation}
respectively. For small values of $\ell$, one can readily utilize Proposition \ref{Prop:ExactZeta}.

\begin{remark}
    The above large $x$-asymptotics of the principal and nonprincipal solutions can be also obtained using a Liouville--Green (also know as WKB) expansion, as the power potential is twice differentiable cf.~\cite[Ch.~6]{Ol97}. We would like to stress that our formulas also hold under the minimal $L^1_{loc}$ regularity assumptions provided the corresponding resolvent are in some Schatten $p$-class. A detailed comparison of our formulas to the classic Liouville--Green analysis is provided in Section~\ref{Sect:LG}. 
\end{remark}
%%%%%%%%%%%%%%%%%%%%%%%%%%%%%%%
%%%%%%%%%%%%%%%%%%%%%%%%%%%%%%%

\subsection{Construction of principal solutions of minimal order}\label{Sect::Principal}
Next we will prove the existence of solutions $\varphi_a(z,x)$ of $\tau f = zf$ which are principal at $x = a$ and entire of minimal order in $z$. Assume that the self-adjoint realizations of $\tau|_{(a,c)}$, $c \in (a,b)$ have eigenvalues with a finite exponent of convergence $\kappa_a$ (this exponent will neither depend on $c$, nor on the particular self-adjoint realization). We will assume that $\cona \geq 1$, otherwise the resolvents are trace class and the construction can be found in \cite{PS24} (see also the previous section). We summarize these assumptions in the following Hypothesis:
\begin{hypothesis}\label{Hypo:nonTrace}
   $(i)$ Hypothesis \ref{h3a}  holds and $x=b$ is in the limit point case.\\[1mm]
    $(ii)$ The Sturm--Liouville differential expression $\tau|_{(a,c)}$ for $c \in (a,b)$ has self-adjoint realizations with non-trace class resolvents, that is $\cona\geq1$.
\end{hypothesis}

We start with the following preparatory lemma. While this result is likely contained somewhere in the literature, we prove it here for the sake of completeness.
\begin{lemma}\label{LemmaEpsilon}
    Assume that $\tau f = \lambda f$ is nonoscillatory at $x = a$, and let $u_a(\lambda, x)$ be a corresponding principal solution at $x = a$. Take any $c \in (a,b)$ with $u_a(\lambda, c) \not = 0$ and denote by $u^{(\varepsilon)}_a(\lambda, x)$ a family of solution of $\tau f = \lambda f$ with $u^{(\varepsilon)}_a(\lambda, a+ \varepsilon) = 0$ and $u^{(\varepsilon)}_a(\lambda, c) \to u_a(\lambda, c)$ for $\e \to 0$ (here we assume $a + \varepsilon < c$). Then for any $x \in (a,b)$ we have $\lim_{\e \to 0} u^{(\e)}_a(\lambda, x) = u_a(\lambda, x)$.  
\end{lemma}
\begin{proof}
    Let $v_a(\lambda, x)$ be any nonprincipal solution of $\tau f = \lambda f$ at $x = a$. Then for $\varepsilon$ small enough $v_a(\lambda, x)$ has no zeros in $(a, a + \varepsilon]$, and hence is linearly independent of $u^{(\varepsilon)}_a(\lambda, x)$. In particular, there are unique coefficients $\alpha^{(\varepsilon)}$, $\beta^{(\varepsilon)}$ such that
    \begin{align} \label{alphabeta}
        \alpha^{(\e)} u^{(\e)}_a(\lambda, x) + \beta^{(\e)} v_a(\lambda, x) = u_a(\lambda, x).
    \end{align}
    By evaluating this equation at $x = a + \e$ we obtain $\beta^{(\e)} = u_a(\lambda, x)/v_a(\lambda, x)$. Hence, we see that as $\e \to 0$, we have that $\beta^{(\e)} \to 0$ (recall that $u_a(\lambda, x)$ is principal). Now, by evaluating \eqref{alphabeta} at $x = c$ we conclude that $\alpha^{(\e)} \to 1$ as $\e \to 0$. This finishes the proof.
\end{proof}
As this is more natural for our main application, we will construct the entire principal solution at the left endpoint $x = a$. As before, $\lambda_j(c,d)$ denotes the $j^{th}$ Dirichlet eigenvalue corresponding to $\tau|_{(c,d)}$, and $\zeta(\ell; (c,d)) = \sum_{j \in \bbN} \lambda_j(c,d)^{-\ell}$. Then we have (cf.~the arguments after Hyp.~\ref{Hypo:Trace})
\begin{align}
    \widetilde \varphi^{(\varepsilon)}_a(z,x) &= \widetilde \varphi^{(\varepsilon)}_a(0,x)\prod_{j \in \bbN} \Big(1 - \frac{z}{\lambda_j(a+\eps, x)} \Big) \notag \\
    &= \widetilde \varphi^{(\eps)}_a(0,x)\exp\Big\lbrace -\sum_{\ell=1}^{\cona} \frac{z^\ell}{\ell} \zeta(\ell; (a+\eps, x)\Big\rbrace\prod_{j\in \bbN} E\Big(\frac{z}{\lambda_j(a+\eps, x)} ,\cona\Big), 
\end{align}
with $\widetilde \varphi^{(\eps)}_a(0, a+\eps) = 0$ and $c$ is chosen such that $\widetilde \varphi^{(\eps)}_a(0, c) \not = 0$ for all $\eps$ sufficiently small. With this is mind, we normalize $\widetilde \varphi^{(\eps)}_a$ such that $\widetilde \varphi^{(\eps)}_a(0, c) = 1$ (for $\eps$ sufficiently small).

To guarantee the convergence to a principal solution near $x = a$ as $\eps \to 0$, we need to make sure, according to Lemma \ref{LemmaEpsilon}, that the rescaled $\varphi^{(\eps)}_a(z,x)$ remains bounded at, say, $x = c$. For this, we assume that $0$ is not a Friedrichs eigenvalue for $\tau|_{(a,c)}$, in particular $\lambda_j(a+\varepsilon', c) \not = 0$ for all $\varepsilon' \in [0,\varepsilon]$ (otherwise shift $c$). We now define
\begin{align}
    \varphi^{(\eps)}_a(z,x) = \widetilde \varphi^{(\eps)}_a(z,x)\exp\Big\lbrace \sum_{\ell=1}^{\cona} \frac{z^\ell}{\ell} \zeta(\ell; (a+\eps, c)\Big\rbrace.
\end{align}
Notice that $\varphi^{(\eps)}_a(z,c) = \underbrace{\widetilde \varphi^{(\eps)}_a(0,c)}_{=1} \prod_{j=1}^\infty E\Big(\frac{z}{\lambda_j(a+\eps, c)},\cona\Big) \to \prod_{j=1}^\infty E\Big(\frac{z}{\lambda_j(a, c)},\cona\Big)$ as $\eps \to 0$. Hence by Lemma \ref{LemmaEpsilon} we can define
\begin{align}\nonumber
    \varphi_a(z,x) &= \lim_{\eps \to 0} \varphi^{(\eps)}_a(z,x)
    \\\label{ConstPhi}
    &= \lim_{\eps \to 0} \varphi^{(\eps)}_a(0,x)\exp\Big\lbrace \sum_{\ell=1}^{\cona} \frac{z^\ell}{\ell} \big[\zeta(\ell; (a+\eps, c)) -\zeta(\ell; (a+\eps, x))\big]\Big\rbrace\prod_{j\in \bbN} E\Big(\frac{z}{\lambda_j(a+\eps, x)} ,\cona\Big) 
    \nonumber \\
    &= \varphi_a(0,x) \prod_{j\in \bbN} E\Big(\frac{z}{\lambda_j(a, x)} ,\cona\Big) \exp\Big\lbrace \sum_{\ell=1}^{\cona} \lim_{\eps \to 0}\frac{z^\ell}{\ell} \big[\zeta(\ell; (a+\eps, c)) -\zeta(\ell; (a+\eps, x))\big]\Big\rbrace.
\end{align}
In particular, Lemma \ref{LemmaEpsilon} implies that the limit $\lim_{\eps \to 0}\zeta(\ell; (a+\eps, c)) -\zeta(\ell; (a+\eps, x))$ must exist. For now, if we denote $\widetilde \zeta(\ell; (x,c)) = \lim_{\eps \to 0} \zeta(\ell; (a+\eps, c)) -\zeta(\ell; (a+\eps, x))$ then we have the analog of \eqref{NormTheta} for principal solutions
\begin{align}\label{NormPhi}
    \lim_{x \downarrow a} \frac{\varphi_a(z,x)}{\varphi_a(0,x)}\exp{\Big\lbrace-\sum_{\ell=1}^{\cona} \frac{z^\ell}{\ell} \widetilde \zeta(\ell; (x,c))\Big\rbrace} = 1, 
\end{align}
which follows from the representation \eqref{ConstPhi} and the fact that $\lambda_j(a,x) \to \infty$ as $x \downarrow a$. Note that the above formula contains $\widetilde \zeta(\ell; (x,c))$ instead of the actual partial $\zeta$-values $\zeta(\ell; (x,c))$. However, it turns out that both are equivalent in the sense that $\lim_{x\downarrow a}  \zeta(\ell; (x,c)) - \widetilde \zeta(\ell; (x,c)) = 0$ for $\ell \in \mathbb N$ as we show in the following proposition.
\begin{proposition}\label{Prop:RankOne}
    We have that 
    \begin{align}\label{Zeta3}
       \lim_{x\downarrow a} \Big\lbrace\zeta(\ell, (x,c)) - \lim_{\varepsilon \to 0} \big[\zeta(\ell; (a+\varepsilon, c)) - \zeta(\ell; (a+\varepsilon, x))\big] \Big \rbrace = 0.
    \end{align}
\end{proposition}
\begin{proof}
    Let $T_{(x,c)}$ denote the Friedrich realization of $\tau|_{(x,c)}$ and similarly for the other intervals. Consider $\widetilde T_{(a+\varepsilon, c)} = T_{(a+\varepsilon, x)} \oplus T_{(x, c)}$ (here we allow $\varepsilon = 0$). Observe that $\sigma(\widetilde T_{(a+\varepsilon, c)}) = \sigma(T_{(a+\varepsilon, x)}) \sqcup \sigma(T_{(x, c)})$. In fact, typical  eigenfunctions $\widetilde \psi_n(x)$ of $\widetilde T_{(a+\varepsilon, c)}$ are of the form
    \begin{align}
        \psi_n(x) = \begin{cases}
            \psi_{(a+\varepsilon, x),j}(x), & x \in (a+\varepsilon, x),
            \\
            0, & x \in (x, c),
        \end{cases} \qquad \text{or} \qquad \psi_n(x) = 
        \begin{cases}
            0, & x \in (a+\varepsilon, x),
            \\
        \psi_{(x, c),k}, & x \in (x, c),
        \end{cases}
    \end{align}
    where $\psi_{(a+\varepsilon, x),j}(x)$ resp.~$\psi_{(x, c),k}(x)$ are eigenfunctions of $T_{(a+\varepsilon, x)}$ resp.~$T_{(x, c)}$ (combination of both cases can happen in case of eigenvalues of multiplicity two).  
    
    We will now deviate a little from our usual convention and denote by $\widetilde \lambda_j(\varepsilon, x)$ the $j^{th}$ eigenvalue of $\widetilde T_{(a+\varepsilon, c)}$ and by $\lambda_j(\varepsilon)$ the $j^{th}$ eigenvalue of $T|_{(a+\varepsilon, c)}$. From $\sigma(\widetilde T_{(a+\varepsilon, c)}) = \sigma(T_{(a+\varepsilon, x)}) \sqcup \sigma(T_{(x, c)})$ it follows that
    \begin{align}
        \zeta(\ell; (a+\varepsilon, x)) + \zeta(\ell; (x,c)) = \sum_{j \geq 1} \widetilde \lambda_j^{-\ell}(\varepsilon, x).
    \end{align}
    With this in mind we can write
    \begin{align}
        \zeta(\ell; (a+\varepsilon, x)) + \zeta(\ell; (x,c)) - \zeta(\ell; (a+\varepsilon, c)) &= \sum_{j = 1}^N \Big[\widetilde \lambda_j^{-\ell}(\varepsilon, x) - \lambda_j^{-\ell}(\varepsilon) \Big] + \sum_{j = N+1}^\infty \Big[\widetilde \lambda_j^{-\ell}(\varepsilon, x) -\lambda_j^{-\ell}(\varepsilon)\Big]\notag
        \\
        &= K(\varepsilon, x) + L(\varepsilon, x).
    \end{align}
    Clearly, for $K$ there is no issue with taking the limit: $\lim_{\varepsilon \to 0} K(\varepsilon, x) = K(x) = \sum_{j=1}^N \big[ \widetilde \lambda_j^{-\ell}(0, x) - \lambda_j^{-\ell}(0) \big]$. Moreover, as $x \downarrow a$, all eigenvalues of $T_{(a,x)}$ diverge to $\infty$, while the $j^{th}$ eigenvalue of $T_{(x, c)}$ converges to the $j^{th}$ eigenvalue of $T_{(a,c)}$. It follows that for each $j$ we have $\lim_{x\downarrow a}\widetilde \lambda_j(0, x) = \lambda_j(0)$. In particular, $\lim_{x\downarrow a} K(x) = 0$. 

To show that $\lim_{x\downarrow a} \big[\lim_{\varepsilon \to 0} L(\varepsilon, x)\big] = 0$ (recall that we already know that the $\varepsilon$-limit exists), we will rely on the fact that $\widetilde T_{(a+\varepsilon, c)}$ differs from $T_{(a+\varepsilon, c)}$ by a rank one resolvent perturbation. Denote by $N_{(a+\varepsilon, c)}(\lambda) = \# \lbrace \lambda_j(\varepsilon) \leq \lambda \colon \lambda_j(\varepsilon) \in \sigma(T_{(a+\varepsilon, c)}) \rbrace$ (resp.~$\widetilde N_{(a+\varepsilon, c)}(\lambda) = ...)$ the eigenvalue counting function (counting multiplicities) for $T_{(a+\varepsilon, c)}$ (resp.~$\widetilde T_{(a+\varepsilon, c)}$). Then we have $|N_{(a+\varepsilon, c)}(\lambda) - \widetilde N_{(a+\varepsilon, c)}(\lambda)| \leq 1$, see \cite[Cor.~II.2.1]{GK69}. This means that
\begin{align}
    \widetilde \lambda_{j+1}(\varepsilon, x)\geq \lambda_j(\varepsilon) \geq \widetilde \lambda_{j-1}(\varepsilon, x),  \qquad \lambda_{j+1}(\varepsilon)\geq \widetilde \lambda_j(\varepsilon, x) \geq \lambda_{j-1}(\varepsilon).
\end{align}
Assuming $N$ is chosen large enough such that all eigenvalues are positive, it follows that 
\begin{align}
    \sum_{j=N+1}^\infty \underbrace{\widetilde \lambda_j^{-\ell}(\varepsilon, x) - \lambda_{j-1}^{-\ell}(\varepsilon)}_{\leq 0} = L(\varepsilon, x) - \lambda_N^{-\ell}(\varepsilon),
    \end{align}
    and 
    \begin{align}
        \sum_{j=N+1}^\infty \underbrace{\widetilde \lambda_j^{-\ell}(\varepsilon, x) - \lambda_{j+1}^{-\ell}(\varepsilon)}_{\geq 0} = L(\varepsilon, x) + \lambda_{N+1}^{-\ell}(\varepsilon).
   \end{align}
   Thus, $L(\varepsilon, x) \in [-\lambda_{N+1}^{-\ell}(\varepsilon), \lambda_N^{-\ell}(\varepsilon)]$, and in particular $\lim_{\varepsilon \to 0} L(\varepsilon, x) \in [-\lambda_{N+1}^{-\ell}(0), \lambda_N^{-\ell}(0)]$. But $\lambda_j(0)$ are just the eigenvalues of $T|_{(a,c)}$ which diverge to $\infty$ as $j \to \infty$. As $N$ can be arbitrary large, it follows that the limit in \eqref{Zeta3} must vanish, finishing the proof.   
\end{proof}
In particular, the above results show that \eqref{asymP} holds. In fact, we can even show that, unlike nonprincipal solutions, the order of entire principal solutions is strongly related to the behavior of the limit \eqref{NormPhi}:
\begin{corollary}\label{CorPhiTilde}
    Assume that $\varphi_a(z,x)$ is an arbitrary entire solution which is principal at $x = a$. Then $\varphi_a(z,x)$ is of minimal order $\kappa_a$ if and only if
    \begin{align}\label{PhiLimit}
        \lim_{x \downarrow a} \frac{\varphi_a(z,x)}{\varphi_a(0,x)}\exp{\Big\lbrace-\sum_{\ell=1}^{\cona} \frac{z^\ell}{\ell} \zeta(\ell; (x,c))\Big\rbrace} = e^{g_{\lfloor \kappa_a \rfloor}(z)}, 
    \end{align}
    for some polynomial $g_{\lfloor \kappa_a \rfloor}(z)$ of degree at most $\lfloor \kappa_a \rfloor$. 
\end{corollary}
\begin{proof}
It is apparent from the representation \eqref{ConstPhi} that the $\varphi_a$ constructed therein is of minimal order $\kappa_a$. Any other entire principal solution of minimal order must be of the form $\widetilde \varphi_a(z,x) = e^{\widetilde g_{\lfloor \kappa_a \rfloor}(z)} \varphi_a(z,x)$ for some polynomial $\widetilde g_{\lfloor \kappa_a \rfloor}$ of degree at most $\lfloor \kappa_a \rfloor$, and the claim immediately follows with $g_{\lfloor \kappa_a \rfloor}(z) = \widetilde g_{\lfloor \kappa_a \rfloor}(z) - \widetilde g_{\lfloor \kappa_a \rfloor}(0)$.
\end{proof}
We also have the analog of Corollary \ref{CorTheta}:
\begin{corollary}\label{CorVarphi}
    Assume that the self-adjoint realizations of $\tau|_{(a,c)}$ have resolvents in the Schatten $p$-class with $p = \cona+1$, but not $p = \cona$. Let $u_a(z_j,x)$ and $u_a(w_j, x)$, $1 \leq j \leq n$ denote the principal solutions at $x = a$. Then the limit 
    \begin{align}\label{uRatio}
        \lim_{x \downarrow a} \frac{u_a(z_1, x) \dots u_a(z_n, x)}{u_a(w_1, x) \dots u_a(w_n, x)} 
    \end{align}
    exists and is nonzero if and only if $\sum_{j =1}^n z_j^\ell = \sum_{j=1}^n w_j^\ell$ for $\ell = 1, \dots ,\cona$. If $\varphi_a(z,x)$ is constructed as in \eqref{ConstPhi}, we even have
    \begin{align} \label{PhiLimit2}
        \lim_{x \downarrow a} \frac{\varphi_a(z_1, x) \dots \varphi_a(z_n, x)}{\varphi_a(w_1, x) \dots \varphi_a(w_n, x)} = 1.
    \end{align}
\end{corollary}
\begin{proof}
The proof follows the same lines as the proof of Corollary \ref{CorTheta}. 
\end{proof}
We finish this section with the following summary of conditions equivalent to $\varphi_a$ being of minimal order.
\begin{corollary}\label{Coriii}
    Let $\varphi_a(z,x)$ be an entire solution principal at $x = a$. Then the following are equivalent:\\[1mm]
        $(i)$ $\varphi_a(z,x)$ is of minimal order $\kappa_a$ in $z$;\\[1mm]
        $(ii)$ $\varphi_a(z,x)$ satisfies \eqref{PhiLimit};\\[1mm]
        $(iii)$ $\varphi_a(z,x)$ satisfies $\lim_{x \downarrow a} \frac{\varphi_a(z_1, x) \dots \varphi_a(z_n, x)}{\varphi_a(w_1, x) \dots \varphi_a(w_n, x)} = 1$, whenever $\sum_{j =1}^n z_j^\ell = \sum_{j=1}^n w_j^\ell$ for $\ell = 1, \dots , \lfloor \kappa_a \rfloor$.
\end{corollary}
\begin{proof}
The equivalence of $(i)$ and $(ii)$ is the content of Corollary \ref{CorPhiTilde}. The equivalence of $(ii)$ and $(iii)$ is shown using the same arguments as in the proof of Corollary \ref{CorTildeTheta}. 
\end{proof}
\begin{remark}\label{RemarkPhiKappa}
Note that the requirement $\sum_{j =1}^n z_j^\ell = \sum_{j=1}^n w_j^\ell$ for $\ell = 1, \dots , \lfloor \kappa_a \rfloor$ in item $(iii)$ is formally distinct from the requirement found in Corollary \ref{CorVarphi}, where we only have $\ell = 1, \dots, \cona$ (recall that $\cona \in \lbrace \lfloor \kappa_a \rfloor - 1, \lfloor \kappa_a \rfloor \rbrace$). The discrepancy in the case of $\cona = \lfloor \kappa_a \rfloor - 1$ (which can only happen if $\kappa_a$ is an integer, i.e, $\lfloor \kappa_a \rfloor = \kappa_a$) comes from the different motivations for the two statements:
\begin{itemize}
    \item On one hand, $\lim_{x \downarrow a} \zeta(\lfloor \kappa_a \rfloor; (x,c)) < \infty$, implying that $\zeta(\lfloor \kappa_a \rfloor; (x,c))$ does \emph{not have to} be included in the regularizing sum \eqref{PhiLimit}, so we do not include it.
    \item On the other hand, as the order of $\varphi_a$ is $\kappa_a = \lfloor \kappa_a \rfloor$, multiplying it with $e^{z^{\lfloor \kappa_a \rfloor}}$ will not destroy the minimal order, so we have to include this possibility to obtain an if and only if statement.
\end{itemize}
 To the best of our knowledge, the case $\cona = \lfloor \kappa_a \rfloor - 1$ does not appear in any explicit examples.
\end{remark}
\subsection{Construction of the characteristic function of minimal order}\label{Sect:Char}

To finalize the construction of a characteristic function of minimal order, we will need the following lemma:
\begin{lemma}\label{Lemma1-3}
    Let $F_{\e}$ be a family of entire functions of order $< k \in \mathbb N$, such that $F_\e(z)$ is continuous with respect to the parameter $\e \in (0, 1)$ for any fixed $z \in \C$. Furthermore, assume that \\[1mm]
        $(i)$ The pointwise limit $F(z) = \lim_{\e \to 0} F_\e(z)$ exists and is an entire function;\\[1mm]
        $(ii)$ The zeros $z_j^{(\e)}$ of $F_\e(z)$ converge to the zeros of $F(z)$, that is, $z_j^{(\e)} \to z_j$ as $\e \to 0$ for an appropriate enumeration;\\[1mm]
        $(iii)$ There exists a constant $C > 0$ such that for any choice of $\e_j \in (0,1)$ we have $\sum_{j \in \N} \frac{1}{|z_j^{(\e_j)}|^k} < C$.

    \noindent Then $F(z)$ is an entire function of order $\leq max \lbrace k-1, \kappa \rbrace$, where $\kappa$ is the exponent of convergence of the sequence $z_j$.
\end{lemma}
\begin{proof}
For simplicity we assume that $z_j^{(\e)}, z_j \not = 0$ for any $j \in \N$. Note that item $(iii)$ implies, by Fatou's Lemma, that $\sum_{j \in \N} \frac{1}{|z_j|^k} < C$. Hence, we can use the Weierstrass factorization theorem to write
\begin{align}
    F_\e(z) = e^{g_\e(z)} \prod_{j \in \N} E\Big(\frac{z}{z_j^{(\e)}},k-1\Big), \quad F(z) = e^{f(z)} \prod_{j \in \N} E\Big(\frac{z}{z_j},k-1\Big),
\end{align}
where $g_\varepsilon(z)$ is, for each $\varepsilon \in (0,1)$, a polynomial of degree at most $k-1$. We know that $f(z)$ is an entire function, but we need to show that it is in fact a polynomial of degree at most $k-1$. To this end, observe that item $(iii)$ implies that
\begin{align}
    \prod_{j \in \N} E\Big(\frac{z}{z_j^{(\e)}},k-1\Big) \to \prod_{j \in \N} E\Big(\frac{z}{z_j},k-1\Big), \quad \text{as} \ \  \e \to 0,
\end{align}
locally uniformly for $z \in \C$. Thus it follows that $e^{g_\e(z)} \to e^{f(z)}$ for $z \in \C \setminus \lbrace z_j \rbrace_{j \in \bbN}$. In particular, for infinitely many values of $z$, we have that $g_\e(z) \to f(z) \mod 2\pi i$ as $\e \to 0$. Furthermore, we can assume that $g_\e(z)$ is continuous with respect to $\e$ (recall that $F_\e(z)$ is continuous with respect to $\e$), as any discontinuity would only amount to a global jump by an integer multiple of $2 \pi i$. Take now $k$ distinct points $w_1, \dots w_{k} \in \C \setminus \lbrace z_j \rbrace_{j \in \bbN}$. It follows that there must be integers $n_{\ell} \in \Z$ such that $g_\e(w_\ell) \to f(w_\ell) + 2\pi i n_\ell$. But the values $g_\e(w_\ell)$, $\ell = 1, \dots, k$ already fully determine the polynomial $g_\e(z)$ (the corresponding Vandermonde matrix is invertible as the points $w_\ell$ are distinct). In particular $g_\e(z) \to g_0(z)$, where $g_0(z)$ is some polynomial of degree $k-1$, and necessarily $g_0(z) = f(z)$ up to some global multiple of $2 \pi n$ (in particular, $n_\ell = n$ are all the same). This finishes the proof.
\end{proof}
We now proceed with the construction of a characteristic function of minimal order. For simplicity, we assume that $0$ is not an eigenvalue of the Friedrichs extension corresponding to $\tau$, otherwise we just perform a spectral shift. Take an entire nonprincipal solution $\theta_b(z,x)$ satisfying the normalization \eqref{NormTheta} (for some fixed $c \in (a,b))$, and a principal entire solution $\varphi_a(z,x)$ satisfying any of the conditions in Corollary \ref{Coriii}. In particular, $\varphi_a(z,x)$ will be of minimal order $\kappa_a$ as a function of $z$, while $\theta_b(z,x)$ could, in theory, be even of infinite order, see Remark \ref{remarkTheta}. Because of \eqref{NormTheta} we then have
\begin{align}
    \lim_{x \uparrow b}  \frac{\varphi_a(z,x)}{\theta_b(z,x)} = \lim_{x \uparrow b} \frac{\varphi_a(z,x)}{\theta_b(0,x)} \exp{\Big\lbrace\sum_{\ell=1}^{\conb} \frac{z^\ell}{\ell} \zeta(\ell; (c,x))\Big\rbrace} = F_0(z).
\end{align}
We claim that $F_0(z)$ is a characteristic function of minimal order. In fact, due to the presence of a purely discrete spectrum, we know that $\displaystyle\lim_{x \uparrow b} \tfrac{\varphi_a(z,x)}{\theta_b(z,x)} = 0$ if and only if $z \in \sigma(T_0)$, where $T_0$ is the Friedrichs extension.  That $F_0(z)$ is of minimal order $\kappa$ follows from Lemma \ref{Lemma1-3} and the fact that 
\begin{equation}
\frac{\varphi_a(z,x)}{\theta_b(0,x)} \exp{\Big\lbrace\sum_{\ell=1}^{\conb} \frac{z^\ell}{\ell} \zeta(\ell; (c,x))\Big\rbrace}
\end{equation}
is of order $\rho = \max \lbrace \kappa_a, \conb \rbrace \leq \kappa$ for each fixed $x \in (a,b)$. We summarize these observations in the following theorem.
\begin{theorem}\label{TheoremChar}
Assume Hypothesis \ref{h3a} holds (in particular the exponent of convergence $\kappa$ is finite), $0\in\rho(T_0)$, and let $c_1,c_2\in(a,b)$. Choose entire solutions $\varphi_a(z,x)$ and $\theta_b(z,x)$ of $\tau f=z f$ which are principal at $x=a$ and nonprincipal at $x=b$, respectively, and satisfy the normalizations
\begin{align}
&\lim_{x \downarrow a} \frac{\varphi_a(z, x)}{\varphi_a(0, x)} \exp{\Big\lbrace-\sum_{\ell=1}^{\cona} \frac{z^\ell}{\ell} \zeta(\ell; (x,c_1))\Big\rbrace} = 1, \quad \lim_{x \uparrow b} \frac{\theta_b(z, x)}{\theta_b(0, x)} \exp{\Big\lbrace\sum_{\ell=1}^{\conb} \frac{z^\ell }{\ell} \zeta(\ell; (c_2,x))\Big\rbrace} = 1. \label{7.1}
\end{align}
Then a characteristic function for $T_0$ with minimal growth order, $\kappa$, is given by
\begin{equation}\label{7.3}
F_0(z)=\lim_{x\uparrow b}\frac{\varphi_a(z,x)}{\theta_b(z,x)},\quad z\in \bbC.
\end{equation}
Moreover, the logarithm of $F_0$ and its derivative, used in Theorem \ref{Thrm3.4}, can be expressed via the limits
\begin{align}\label{LogF}
\ln \, F_0(z) &= \lim_{x\uparrow b}\Big[\ln \, \varphi_a(z,x) - \ln \, \theta_b(0,x) + \sum_{\ell=1}^{\conb} \frac{z^{\ell}}{\ell} \zeta(\ell; (c,x))\Big],
\\\label{dLogF}
\frac{d}{dz}\ln \, F_0(z) &= \lim_{x\uparrow b}\Big[\frac{d}{dz}\ln \, \varphi_a(z,x) + \sum_{\ell=1}^{\conb} z^{\ell-1} \zeta(\ell; (c,x))\Big]. 
\end{align}
An analogous statement holds with the roles of the endpoints $x=a$ and $x=b$ interchanged.
\end{theorem}
\begin{remark}
We have shown in \cite[Thm.~3.3]{FPS25a} that the analytic continuation of the spectral $\zeta$-function to $\Re \, z \leq \kappa$ is intimately related to an asymptotic expansion of the form 
\begin{align}
    \ln\,F_0(z)= \sum_{j=0}^N  \sum_{k=0}^{M} d_{j,k} z^{\kappa-(j/m)}\ln^{k}z+o\big(z^{\kappa-(N/m)-\delta}\big),\quad z\to\infty,
\end{align}
in a sector away from the real line. Moreover, only the terms $d_{j,k} \not = 0$ such that $k \geq 1$ or $\kappa-(N/m) \not \in \mathbb Z \setminus \lbrace 0 \rbrace$ contribute to poles in the analytic continuation. It follows that the $z^\ell$-terms in \eqref{LogF} coming from the regularization via the partial $\zeta$-values, do not affect the pole structure. This suggests that the pole structure of the analytic continuation of the spectral $\zeta$-function is completely specified by the large $z$-behavior of $\varphi_a(z,x)$ as $x \uparrow b$.
\end{remark}

This theorem holds even when the resolvents of the truncated problems are trace class. In that case, the partial $\zeta$-values are not present, and we recover the normalization used for the definition of the regularization index in \cite{PS24}. Moreover, \eqref{7.1} can be understood as the extension of the typical normalization using generalized boundary values in the limit circle case to the limit point with resolvent in some Schatten $p$-class. The role that these normalizations play is to guarantee that the characteristic function constructed is of minimal growth order in the spectral parameter $z$.

In addition, we can replace the limit $1$ in the two limits \eqref{7.1} with $e^{g_{\lfloor\kappa\rfloor, j}(z)}$, $j = 1,2$, where $g_{\lfloor\kappa\rfloor, j}$ are polynomials of degree at most $\lfloor\kappa\rfloor$. This is not surprising, as entire functions of minimal order $\kappa$ are only unique up to such prefactors. Finally, $0\in\rho(T_0)$ can be replaced with some $z_1\in\rho(T_0)$ as previously described (see the sentence after \eqref{NormTheta}).

Note that if either endpoint (or both) is in the limit circle case, then principal solutions at these endpoints will satisfy Friedrichs boundary conditions, giving us the Friedrichs extension $T_0$. Other self-adjoint extensions can be treated along the lines described in Remark \ref{Rem:3.3}.

\section{Further applications}\label{Sect:Applications}

Let us now consider two more applications of the construction of Sections \ref{Sect::NonPrincipal} and \ref{Sect::Principal}. One is based on a comparison of \eqref{NormTheta} and \eqref{NormPhi} with the Liouville--Green expansion for Schr\"odinger operators in the exponentially decaying/growing region, while the other relates to the convergence rate of the Dirichlet eigenvalues of the truncated problem, that is, the rate of convergence of  $\lambda_j(a+\eps, c) \to \lambda_j(a,c)$ (or $\lambda(c, b - \eps) \to \lambda(c, b)$).

\subsection{Connections to Liouville--Green expansions}\label{Sect:LG}
In this section we will consider the classic setting of a Schr\"odinger operator $\tau = -\frac{d^2}{dx^2} + q$ on $(a, \infty)$ with a confining potential $q(x) \to \infty$ as $x \to \infty$. We set $b = \infty$, since otherwise the self-adjoint realizations would have Weyl asymptotics, that is, $\lambda_j \propto j^2$, and we would be in the trace class resolvent case. Note, however, that both $a = -\infty$ and $a \in \bbR$ are allowed. While we could also work with a more general three-coefficient Sturm--Liouville differential expression, we have choosen to consider a Schr\"odinger differential expression for the sake of simplicity. 

We will be interested in the classically `forbidden' region, where principal solutions typically decay exponentially, while nonprincipal solutions grow exponentially (in particular there are no oscillations). For a study of Liouville--Green asymptotics in the classically \emph{allowed} region and very rough decaying potentials, see the recent work \cite{LW25}.
\begin{hypothesis}[Liouville--Green asymptotics]\label{Hypo::LG}
    Assume that the Schr\"odinger differential expression $\tau = -\frac{d^2}{dx^2} + q$, with $x \in (a,\infty)$ and $q(x) \to \infty$ for $x \to \infty$ has resolvents in the Schatten $p$-class with $p = \coninf+1$, but not $p = \coninf$. Furthermore, assume that principal (nonprincipal) solutions $u_\infty(\lambda, x)$ ($v_\infty(\lambda, x)$) have the asymptotics
\begin{align}\label{LG_asymp}
        u_\infty(\lambda, x) \propto \frac{1}{q^{1/4}(x)} \exp\Big \lbrace\int_{c}^x - \sqrt{q(t) - \lambda} \, dt \Big\rbrace, \quad  \Bigg(v_\infty(\lambda, x) \propto \frac{1}{q^{1/4}(x)} \exp \Big \lbrace\int_{c}^x  \sqrt{q(t) - \lambda} \, dt  \Big \rbrace\Bigg)
    \end{align}
    for $x \to \infty$ and $\lambda \in \bbR$.
\end{hypothesis}
The choice of $c \in (a, \infty)$ is made such that $q(t) - \lambda \geq \varepsilon > 0$ for $t \geq c$, but is otherwise irrelevant to the asymptotics. We will later see that under Hypothesis \ref{Hypo::LG} the above asymptotic formulas will even hold if we replace $\lambda$ by $z \in \bbC$.
\begin{remark}
General conditions which guarantee that \eqref{LG_asymp} holds can be found in \cite[Ch.~6]{Ol97}, which involve certain continuity and differentiability assumptions on the potential $q$ (see e.g.~\cite[Ch.~6, Thm.~2.1]{Ol97}). As we work with rougher $L^1_{loc}$ potentials, we instead choose to include the Liouville--Green asymptotics (i.e., the leading term in the Liouville--Green expansion) simply as an assumption in Hypothesis \ref{Hypo::LG}. 
\end{remark}

Next, assuming that $|q(x)| > (1 + \eps)|\lambda|$ for some $\eps > 0$ and all $x > c$, we have the convergent power series expansion
\begin{align}
    \sqrt{q(x) - z} = \sqrt{q(x)} \Bigg(1 - \sum_{\ell = 1}^\infty \Big(\frac{z}{q(x)}\Big)^\ell \frac{(2\ell-3)!!}{2^\ell \ell!} \Bigg), \quad (-1)!! := 1.
\end{align}
Due to the local uniform convergence of the series in $x \in (c,\infty)$, we have therefore
\begin{align}
   \int_{c}^x  \sqrt{q(t) - z} \, dt &= \int_{c}^x  \sqrt{q(t)}  \, dt - \sum_{\ell=1}^\infty \frac{z^\ell}{\ell}\int_{c}^x q(t)^{\frac{1}{2}-\ell}\frac{(2\ell-3)!!}{2^\ell (\ell-1)!} \, dt
   \\
   &= \int_{c}^x  \sqrt{q(t)}  \, dt - \sum_{\ell=1}^{\coninf} \frac{z^\ell}{\ell}\int_{c}^x q(t)^{\frac{1}{2}-\ell}\frac{(2\ell-3)!!}{2^\ell (\ell-1)!} \, dt + O\Big(z^{\coninf+1} \int_c^x q(t)^{-\frac{1}{2}-\coninf} \, dt \Big). \notag
\end{align}
In particular, it now becomes apparent that 
\begin{align}
    \frac{v_\infty(\lambda,x)}{v_\infty(0, x)} &\propto \exp \Bigg\lbrace -\sum_{\ell=1}^{\coninf} \frac{\lambda^\ell}{\ell}\int_{c}^x q(t)^{\frac{1}{2}-\ell}\frac{(2\ell-3)!!}{2^\ell (\ell-1)!} \, dt + O\Big(\lambda^{\coninf+1} \int_c^x q(t)^{-\frac{1}{2}-\coninf} \, dt \Big) \Bigg \rbrace \notag
    \\
    &\propto \exp{\Big\lbrace-\sum_{\ell=1}^{\coninf} \frac{\lambda^\ell}{\ell} \zeta(\ell; (c',x))\Big\rbrace},
\end{align}
with the second ``$\propto$" following by \eqref{NormTheta}. By introducing the notation $\xi(\ell; (c,x)) = \frac{(2\ell-3)!!}{2^\ell (\ell-1)!}\int_{c}^x q(t)^{\frac{1}{2}-\ell} dt$, it follows, for any $\lambda \in \bbR$, that the limit
\begin{align}\label{LimLim}
    \lim_{x\to \infty}\Bigg[\sum_{\ell=1}^{\coninf} \frac{\lambda^\ell}{\ell}\Big\lbrace\xi(\ell; (c,x)) -  \zeta(\ell; (c',x))\Big\rbrace  + O\Big(\lambda^{\coninf+1} \int_c^x q(t)^{-\frac{1}{2}-\coninf} \, dt \Big)
    \Bigg] \in \mathbb R
\end{align}
must exist. Recall that $c$ has to be chosen close enough to $\infty$ such that $|q(x)| > (1 + \eps)|\lambda|$ for $x \in (c,\infty)$ and some $\eps > 0$, so that the infinite sum is guaranteed to converge. No requirements are necessary for $c'\in (a,\infty)$. 

In case $\int_c^\infty q(t)^{-\frac{1}{2} - \coninf} dt < \infty$, it is easy to see that the existence of the limit \eqref{LimLim} implies that $\xi(\ell; (c,x)) = \zeta(\ell; (c', x)) + C_\ell + o(1)$ as $x \to \infty$ for $\ell = 1, \dots,\coninf$. As we now argue, $\int_c^\infty q(t)^{-\frac{1}{2} - \coninf} dt = \infty$ would lead to a contradiction. Assume $\int_c^\infty q(t)^{-\frac{1}{2} - \coninf} dt = \infty$. In that case we have in \eqref{LimLim} that $O\Big(\lambda^{\coninf+1} \int_c^x q(t)^{-\frac{1}{2}-\coninf} \, dt \Big) = \lambda^{\coninf+1}\frac{(2\coninf-1)!!}{2^{\coninf+1} (\coninf+1)!}\int_c^x q(t)^{-\frac{1}{2}-\coninf} \, dt + lower \ order \ terms$ (just add one more term in the Taylor expansion; as $q(x) \to \infty$ for $x \to \infty$, the next remainder will be of smaller order). Then necessarily $|\xi(\ell; (c,x)) - \zeta(\ell; (c', x))| \gtrsim \int_c^x q(t)^{-\frac{1}{2} - \coninf} \, dt$ for at least one $\ell \in \lbrace 1, \dots,\coninf \rbrace$. By choosing $\lambda_1, \dots, \lambda_n$ such that $\sum_{j=1}^n \lambda_j^m = 0$ for all $m \leq \coninf$ except $m = \ell$, that is, $\sum_{j=1}^n \lambda_j^\ell \not =  0$, and summing \eqref{LimLim} with $\lambda = \lambda_j$, we arrive at
\begin{align}
    \lim_{x\to \infty}\Bigg[\Big(\sum_{j=1}^{n} \frac{\lambda_j^\ell}{\ell}\Big)\Big\lbrace\xi(\ell; (c,x)) -  \zeta(\ell; (c',x))\Big\rbrace  + O\Big(\max_j\lbrace |\lambda_j|^{\coninf+1}\rbrace\int_c^x q(t)^{-\frac{1}{2}-\coninf} \, dt \Big)
    \Bigg] \in \bbR.
\end{align}
By scaling $\lambda_j(\eps) = \eps\lambda_j$, we conclude that 
\begin{align}
\lim_{x\to \infty}\Bigg[\Big(\eps^\ell\sum_{j=1}^{n} \frac{\lambda_j^\ell}{\ell}\Big)\underbrace{\Big\lbrace\xi(\ell; (c,x)) -  \zeta(\ell; (c',x))\Big\rbrace}_{|\, \cdot \,| \, \gtrsim \, \int_c^x q(t)^{-\frac{1}{2}-\coninf}} +\eps^{\coninf+1}O\Big(\int_c^x q(t)^{-\frac{1}{2}-\coninf} \, dt \Big) 
    \Bigg] \in \bbR.
\end{align}
By letting $\varepsilon \to 0$ we infer that obtaining a finite limit is possible only if $\int_c^\infty q(t)^{-\frac{1}{2}-\coninf}$ remains finite as $x \to \infty$. We summarize these results in the following theorem:
\begin{theorem}\label{TheoremLG}
    Assume that the Schr\"odinger differential expression $\tau = -\frac{d^2}{dx^2} + q$, with $x \in (a,\infty)$ satisfies Hypothesis \ref{Hypo::LG}. Then for any choice of $c, c'\in (a,\infty)$, and $\ell \in \bbN_+$ we have that\\[1mm]
   $(i)$ $\lim_{x \to \infty} \int_c^x q(t)^{\frac{1}{2} - \ell} dt = \infty$ if and only if $\ell \in \lbrace 1, \dots,\coninf \rbrace$;\\[1mm]
       $(ii)$ $\lim_{x \to \infty} \Big[\frac{(2\ell-3)!!}{2^\ell (\ell-1)!}\int_{c}^x q(t)^{\frac{1}{2}-\ell} dt - \zeta(\ell; (c',x))\Big] \in \bbR$, for $\ell \in \lbrace 1, \dots,\coninf \rbrace$.
\end{theorem}
The above results demonstrate that the asymptotic formulas
\begin{align}
    u_\infty(z,x) \propto u_\infty(0,x) \exp{\Big\lbrace\sum_{\ell=1}^{\coninf} \frac{z^\ell}{\ell} \zeta(\ell; (c,x))\Big\rbrace}, \quad
    v_\infty(z,x) \propto v_\infty(0,x) \exp{\Big\lbrace-\sum_{\ell=1}^{\coninf} \frac{z^\ell}{\ell} \zeta(\ell; (c,x))\Big\rbrace},
\end{align}
 for $x \to \infty$ can be viewed as a natural generalization of the Liouville--Green asymptotic formula in the classically forbidden region. As alluded to earlier, such formulas usually require regularity assumptions of the potential, see \cite[Thm.~2.1]{Ol97}. No such assumptions are made in the present work, rather we require a purely \emph{spectral} condition given in terms of the associated resolvents being in the Schatten $p$-class with $p = \coninf+1$.
 
Theorem \ref{TheoremLG} also suggests a connection between Liouville--Green type formulas in the classically forbidden regions, and convergence properties of the expressions in $(i)$, $(ii)$ above. This raises the question:  Under which additional assumptions does the converse of Theorem \ref{TheoremLG} hold?
\begin{problem}\label{ProblemLG}
    Let a Schr\"odinger differential expression $\tau = -\frac{d^2}{dx^2} + q$, with $x \in (a,\infty)$ and $q(x) \to \infty$ as $x \to \infty$ be given. Assume that items $(i)$ and $(ii)$ stated in Theorem \ref{TheoremLG} hold. Will Hypothesis
\ref{Hypo::LG} hold in this case or are some additional smoothness assumptions on $q(x)$ necessary?
\end{problem}

\subsection{Approximation of singular problems by regular ones: convergence rate of eigenvalues}\label{Sect:EigenConv}

Next, we revisit a classic problem in the spectral theory of Sturm--Liouville operators. Consider a generally singular Sturm--Liouville expression $\tau$ on $(a,b)$, and a self-adjoint realization $T$ (later we take the Friedrichs realization). As we are interested in the singular case, we assume that $\tau$ is in the limit point case at one (or both) of the endpoints. In particular, $T$ might satisfy a boundary condition at the limit circle endpoint, but will certainly not satisfy any coupled boundary conditions. 

It is known that if we consider appropriate self-adjoint realizations $T_n$ of the truncated differential expression $\tau|_{(a_n, b_n)}$ for $n \in \bbN$ and $a < a_n < b_n < b$, then the spectrum $\sigma(T_n)$ will approximate $\sigma(T)$ in the following sense. For any $\lambda \in \sigma(T)$ there will be an appropriate sequence $j_n$ such that $\lambda_{j_n, n} \to \lambda$. Here $\lambda_{j, n}$ denotes the $j^{th}$ eigenvalue of $T_n$ (note that $\sigma(T_n) = \lbrace \lambda_{j,n} \rbrace_{j=1}^\infty$ is always discrete as $\tau|_{(a_n, b_n)}$ is regular at both endpoints). Moreover, if $T$ has exactly $N \in \bbN \cup \lbrace \infty \rbrace$ discrete eigenvalues $\lbrace \lambda_j \rbrace_{j=1}^N$ below its continuous spectrum, then $\lambda_{j,n} \to \lambda_j$ as $n \to \infty$ for $j \leq N$, and $\lambda_{j,n} \to \inf \sigma_{ess}(T)$ as $n \to \infty$ for $j > N$. Further details, in particular the exact choice of boundary conditions satisfied by the `inherited operator' $T_n$ can be found in \cite{BEWZ,EMZ01}, \cite[Ch.~10.8]{Ze05}. For our purposes, it is enough to state that the inherited boundary condition near a limit point endpoint is the Dirichlet boundary condition.

In the following, we will compare the convergence rates $\lambda_{j, n} - \lambda_j \to 0$ for different $j \in \bbN$, and compute a universal bound for this rate which holds for all $j$. Throughout the analysis, we make the following extra assumptions:
\begin{hypothesis}\label{Hypo::ConvRate}
    Assume that the self-adjoint realizations of $\tau|_{(a,c)}$ for $c \in (a,b)$ have trace-class resolvents, meaning $\cona = 0$ (i.e., we are in the setting of \cite{PS24} at $x = a$), and that the self-adjoint realizations of $\tau$ have resolvents in the Schatten $p$-class with $p = \conb+1$, but not $p = \conb$. We will denote by $\sigma(T) = \lbrace \lambda_j \rbrace_{j \in \bbN}$ the spectrum of the Friedrichs realization $T$ of $\tau$, and by $\sigma(T_{(a,x)}) = \lbrace \lambda_j(x) \rbrace_{j \in \bbN}$ the spectrum of the Friedrichs realization $T|_{(a,x)}$ of $\tau|_{(a,x)}$.
\end{hypothesis}
Note that the above hypothesis is equivalent to Hypothesis \ref{Hypo:Trace}. We state it here for the convenience of the reader. 

We will be interested in the convergence rate $\lambda_j(x) - \lambda_j \to 0$, which holds by the results outlined in \cite{BEWZ}. However, in contrast to \cite{BEWZ}, we will have to keep one endpoint fixed which we choose to be $a$, that is, $a_n = a = const$. Moreover, it is more natural for us to work with a \emph{continuously varying} right endpoint $x \uparrow b$, that is we replace the sequence $b_n$ with a continuous variable $x$. This later point is however just a matter of convention.

Assuming Hypothesis \ref{Hypo::ConvRate}, we know that there exists an entire and principal solution $\varphi_a(z,x)$ at $x = a$ satisfying the normalization $\lim_{x\downarrow a} \frac{\varphi_a(z_1, x)}{\varphi_a(z_2,x)} = 1$ (for an explicit construction see \cite[Sect.~3]{PS24}). As argued in Section \ref{Sect:Trace}, $\varphi_a(z,x)$ can be expressed as $\varphi_a(z,x) = \varphi_a(0,x) \prod_{j \in \bbN} \big(1- \frac{z}{\lambda_j(x)} \big)$. Similarly, at the endpoint $b$, let us choose a \emph{nonprincipal} entire fundamental solution $\theta_b(z,x)$ satisfying the normalization \eqref{NormTheta} with $c = a$ (the choice of $c$ in \eqref{NormTheta} is justified by the trace class resolvent property at $x = a$ postulated in Hypothesis \ref{Hypo::ConvRate}). We can w.l.o.g assume that $0 \not \in \sigma(T)$ (otherwise perform a spectral shift), in which case $\varphi_a(0,x)$ is nonprincipal at $x = b$, and we can even assume that $\theta_b(0,x) = \varphi_a(0,x)$.

Note that by definition $\varphi_a(\lambda_m, x)$ will be the $m^{th}$ eigenfunction of $T$, and therefore principal at both endpoints. In particular we have
\begin{align}\nonumber
   \frac{\varphi_a(\lambda_m, x)}{\theta_b(\lambda_m,x)} \overset{x\uparrow b}{\sim} \frac{\varphi_a(\lambda_m, x)}{\theta_b(0,x)} \exp{\Big\lbrace\sum_{\ell=1}^{\conb} \frac{\lambda_m^\ell}{\ell} \zeta(\ell; (a,x))\Big\rbrace} &= \frac{\varphi_a(\lambda_m, x)}{\varphi_a(0,x)} \exp{\Big\lbrace\sum_{\ell=1}^{\conb} \frac{\lambda_m^\ell}{\ell} \zeta(\ell; (a,x))\Big\rbrace} 
    \\\label{EEE}
    & = \bigg(\prod_{j \in \bbN \setminus \lbrace m \rbrace} E\Big( \frac{\lambda_m}{\lambda_j(x)},\conb \Big) \bigg) E\Big( \frac{\lambda_m}{\lambda_m(x)},\conb \Big).
\end{align}
Note that the infinite product in the last line converges as $x \uparrow b$, while 
\begin{align}
    E\Big( \frac{\lambda_m}{\lambda_m(x)},\conb \Big) \propto \Big(1 - \frac{\lambda_m}{\lambda_m(x)} \Big) \propto \lambda_m(x) -\lambda_m, \quad \text{for } \ x \uparrow b
\end{align}
(recall $\lambda_m \not = 0$ by assumption). This shows that the convergence rate $\lambda_m(x) - \lambda_m \to 0$ is proportional to the convergence rate $\frac{u_b(x)}{v_b(x)} \to 0$ of the ratio of principal and nonprincipal solutions at $x = b$. In light of \eqref{7.1} this leads to a remarkably \emph{rigid} relationship between the convergence rate of $\lambda_m(x) - \lambda_m \to 0$ and $\lambda_j(x) - \lambda_j \to 0$, which is fully specified in terms the divergence rates of the partial $\zeta$-values $\zeta(1; (a,x)), \dots, \zeta(\conb; (a,x))$:
\begin{align}
    \lambda_j(x) - \lambda_j \propto \frac{\varphi_b(\lambda_j, x)}{\theta_b(\lambda_j, x)} &\propto \frac{\varphi_b(\lambda_m, x)}{\theta_b(\lambda_m, x)}\exp{\Big\lbrace2\sum_{\ell=1}^{\conb} \frac{\lambda_j^\ell - \lambda_m^\ell}{\ell} \zeta(\ell; (a,x))\Big\rbrace} \notag 
    \\
    &\propto \big(\lambda_m(x) - \lambda_m\big)\exp{\Big\lbrace2\sum_{\ell=1}^{\conb} \frac{\lambda_j^\ell - \lambda_m^\ell}{\ell} \zeta(\ell; (a,x))\Big\rbrace}
\end{align}
(note that the $\varphi_b$ above instead of the $\varphi_a$ in \eqref{EEE} is not a typo, as both are equal up to a multiplicative constant for $z = \lambda_j$). We summarize these findings in the following proposition.
\begin{proposition}\label{PropConv}
    Assume Hypothesis \ref{Hypo::ConvRate} holds. Then 
    \begin{align}\label{LambdaConv}
        \lambda_j(x) - \lambda_j \propto \big(\lambda_m(x) - \lambda_m\big)\exp{\Big\lbrace2\sum_{\ell=1}^{\conb} \frac{\lambda_j^\ell - \lambda_m^\ell}{\ell} \zeta(\ell; (a,x))\Big\rbrace}, \quad \text{for} \ x \uparrow b. 
    \end{align}
\end{proposition}
Note that Proposition \ref{PropConv} implies that the convergence $\lambda_j(x) - \lambda_j \to 0$ becomes slower as $j$ increases.  

The preceding arguments rely on the endpoint $a$ being `fixed', and only the right endpoint $x \uparrow b$ varying. Additionally, we have to assume that the self-adjoint realizations of $\tau|_{(a,c)}$ have trace class resolvents. It would be interesting to understand whether, in the absence of these assumptions, the convergence rates of $\lambda_j(x_1,x_2) \to \lambda_j$ with $\lambda_j(x_1, x_2)$ being the $j^{th}$ Dirichlet eigenvalue corresponding to $\tau|_{(x_1, x_2)}$ have a similar structure. We pose this as an open problem:
\begin{problem}\label{Prob:Conv}
    Assume that $\tau$ has self-adjoint realizations with resolvents in the Schatten $p$-class. What can be said about the relationship between the convergence rates of 
    \begin{align}
        \lambda_j(x_1,x_2) \to \lambda_j \ \text{as} \ x_1 \downarrow a, \text{and} \ x_2 \uparrow b, 
    \end{align}
    for different values $j \in \bbN$? 
\end{problem}

Next we will show a universal upper bound on the convergence rate $\lambda_j(x) - \lambda_j \to 0$, without any reference to the convergence rate of the other eigenvalues.
\begin{corollary}\label{CorLambdaX}
    Assume Hypothesis \ref{Hypo::ConvRate} holds. Then for any $K>0$ and $j \in \mathbb{N}$, we have
    \begin{align}\label{LambdaK}
        \lambda_j(x) - \lambda_j = o\Big(\exp{\Big\lbrace-K\zeta(1; (a,x))\Big\rbrace} \Big), \quad x \uparrow b.
    \end{align}
\end{corollary}
\begin{proof}
    The claim follows from \eqref{LambdaConv} by choosing $m$ sufficiently large such that $2(\lambda_m - \lambda_j) > K$ and observing that the divergence rate of $\zeta(1; (a,x))$ as $x \uparrow b$ will asymptotically dominate the divergence rate of any $\zeta(\ell; (a,x))$ with $\ell \geq 2$.
\end{proof}

\begin{remark}
Note that there are multiple ways of determining $\zeta(1; (c,x))$ up to a bounded error. In case Hypothesis \ref{Hypo::LG} holds, we can use Theorem \ref{TheoremLG}. More generally, we have 
\begin{align}
        \int_c^x u_b(\lambda,t) v_b(\lambda,t) r(t)dx = \zeta(1; (c,x)) + O(\zeta(2; (c,x))), \quad \text{for} \, x \uparrow b,
\end{align}
for principal and nonprincipal solutions $u_b(\lambda, x)$ and $v_b(\lambda, x)$ satisfying $W(v_b, u_b) \equiv -1$ (the minus is due to the convergence at the \emph{right} endpoint). This formula is shown as part of Lemma \ref{LemmaInt1} (for the left endpoint $a$, but the argument works analogously for $b$).
\end{remark}

Finally, let us clarify that while the right-hand side in \eqref{LambdaK} is the same for any fixed $j \in \bbN$, the convergence will certainly not be uniform in $j$, see Figure~\ref{Fig:Lag} below. To see this, note that $\lambda_j(x)^{-1}$ is summable by the resolvent trace class assumption for $\tau|_{(a,x)}$ , but $\lambda_j^{-1}$ is not if $\conb \geq 1$.
\begin{figure}[ht]
  \centering
  % Left figure minipage
  \begin{minipage}{0.48\textwidth}
    \centering
    \includegraphics[width=0.7\textwidth]{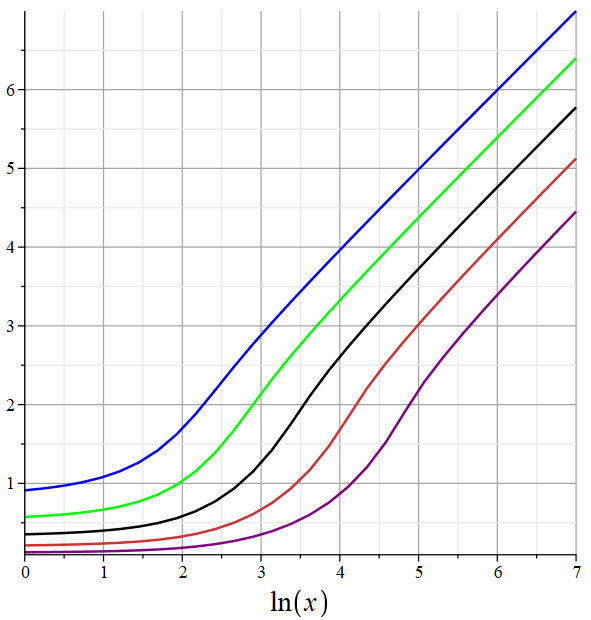}
    % Optional: \caption{Left sub-caption} if you want (a)
  \end{minipage}
  \hfill % Adds horizontal space between the two minipages
  % Right figure minipage
  \begin{minipage}{0.48\textwidth}
    \centering
    \includegraphics[width=.78\textwidth]{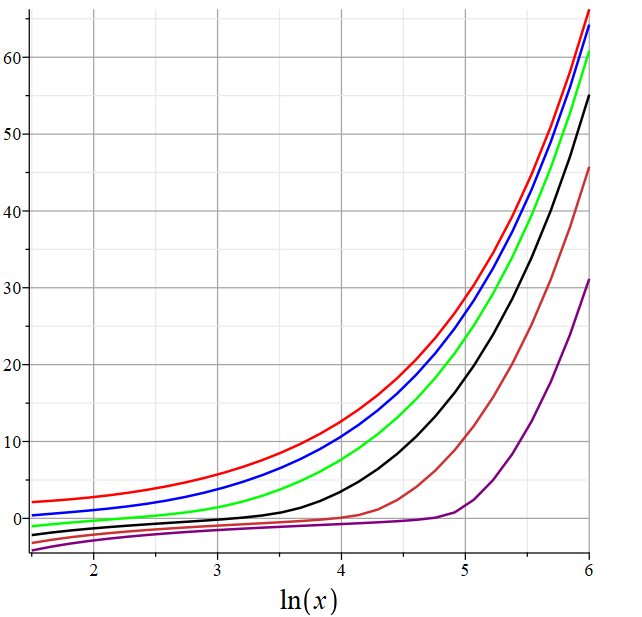}
    % Optional: \caption{Right sub-caption} if you want (b)
  \end{minipage}

 % \vspace{10pt} % Optional: Adds a small vertical gap before the caption
  \caption{\textbf{Left:} Plot of the functions $f_j(x) = \frac{\ln(\lambda_j(x) - \lambda_j) - \ln(\lambda_1(x) - \lambda_1)}{2(\lambda_j - \lambda_1)}$ in the case of the truncated Laguerre problem with $\gamma = 1$, see Section~\ref{LagEx}, with $j = 2$ (blue), $j = 4$ (green), $j = 8$ (black), $j = 16$ (dark red) and $j = 32$ (purple). Here $\lambda_j(x)$ are numerically computed Dirichlet eigenvalues for the problem restricted to $(0,x)$ and $\lambda_j = j-1$. As the problem has Hibert--Schmidt resolvents, and $\zeta(1; (0,x)) = \ln(x) + O(1)$, cf.~Section~\ref{LagEx}, Proposition \ref{PropConv} implies $f_j(x) = \ln(x) + O(1)$, that is,~the growth in the plot becomes linear with slope $1$ (note that the horizontal axis is logarithmic in $x$). \textbf{Right}: Plot of the functions $g_j(x)  =-\frac{\ln(\lambda_j(x) - \lambda_j)}{\ln(x)}$ with $j \in \lbrace 1, 2, 4, 8, 16, 32\rbrace$ for the same problem and the same color code as on the left, with the case $j = 1$ being in bright red. According to Corollary \ref{CorLambdaX} we have that $g_j(x) \to \infty$ as $x \to \infty$.}
  \label{Fig:Lag}
\end{figure} 

\section{Examples}\label{Sect:Examples}

Before turning our attention to some explicit examples, we would like to point out how the current study can be utilized to readily extend previous results. In particular, the P\"oschl--Teller potential recently studied in \cite{FS25} and the generalized Bessel equation studied in \cite{FPS25} can now be extended to include parameter choices which yield limit point endpoints.
By our general results here, this can easily be done by using the characteristic function corresponding to the Friedrichs boundary condition at the endpoint which will become limit point. One then simply allows the parameter choices of the differential expression to vary into the range for which the endpoint becomes limit point in both of those examples. This clearly illustrates that our work above greatly generalizes previous results.

\subsection{Harmonic oscillator}

We revisit, here, the harmonic oscillator $-y''(x)+ x^2 y(x)=z y(x),$ for $x\in\bbR$. The principal and nonprincipal solutions at the limit point endpoints $\pm\infty$ are $\phi_{\pm\infty}(z,x)=U(-z/2,\pm\sqrt{2}x)$ and $\theta_{\pm\infty}(z,x)=V(-z/2,\pm\sqrt{2}x)$, respectively where $U(c,\xi)$ and $V(c,\xi)$ are the parabolic cylinder functions \cite[Sect.~12.2(i)]{DLMF}. The asymptotics of $\theta_{\infty}(z,x)$, $\varphi_{\infty}(z,x)$ for $x \to \infty$ are given by \eqref{HarmV}, \eqref{HarmU} respectively (and similarly for the solutions at $-\infty$). By Corollary \ref{Coriii} it follows that $\varphi_{\pm \infty}(z,x)$ is of minimal order $\kappa_{\pm \infty} = 1$, and $\theta_{\pm\infty}(z,x)$ satisfies \eqref{TTTT} with $\mathfrak{p}_{\pm\infty} = 1$. By utilizing the connection formula \cite[Eq. 12.2.15]{DLMF} 
\begin{equation}\label{eq:c9}
U\left(c,-\xi\right)=-\sin(\pi c)U(c,\xi)+\frac{\pi}{\Gamma
((1/2)+c)}V(c,\xi),
\end{equation}
we find that the characteristic function in Theorem \ref{MainTheorem1} can be written as
\begin{align}
    F_0(z)=\lim_{x\to\infty}\frac{\phi_{-\infty}(z,x)}{\theta_{\infty}(z,x)}=\frac{\pi}{\Gamma((1-z)/2)}.
\end{align}
This agrees with the standard known result of the spectrum being given by $z=2n+1,$ $n\in\bbN_0$.

If one instead studies this problem on the half-line $(0,\infty)$, we can use the same solution as above to now find that the characteristic function is given by
\begin{align}
F_\alpha(z)&=\cos(\a)U(-z/2,0)+\sin(\a)\sqrt{2}U'(-z/2,0)\\
&= \sqrt{\pi}2^{(z-1)/4}\Big[\cos(\alpha) \frac{1}{\Gamma((3-z)/4)} + \sin(\alpha)\frac{2}{\Gamma((1-z)/4)}\Big],\quad \a\in[0,\pi),\ z\in\bbC,\notag
\end{align}
where we used \cite[Eq. 12.2.6 \& 12.2.7]{DLMF} and the prime denotes the derivative with respect to the second argument as usual. Here, the term in the square bracket would itself be a characteristic function of minimal order, as the prefactor never vanishes and is of order $1$. Notice that for the Dirichlet condition $\a=0$, the spectrum is given by $z=4n+3$, $n\in\bbN_0$, while Neumann yields $z=4n+1$, $n\in\bbN_0$. The union of the spectrum of these two operators should equal the full line problem (once again, see \cite{FS25a}), which is readily verified since $\{4n+3\}_{n\in\bbN_0}\cup\{4n+1\}_{n\in\bbN_0}=\{2n+1\}_{n\in\bbN_0}$.

\subsection{Laguerre example}\label{LagEx} We can now revisit the Laguerre example which was briefly mentioned in \cite[Sect.~11.5]{PS24} where it served as a type of \emph{counter-example} to the trace class resolvent case studied therein. As in the previous example, the eigenvalue growth for this problem is \emph{linear}, hence resolvents will no longer be trace class, but rather Hilbert--Schmidt.  

The Laguerre eigenvalue equation has the form
\begin{align}
    \tau_{Lag, \gamma} y = -\frac{1}{e^{-x} x^{\gamma-1}}\left(x^{\gamma}e^{-x}y'\right)'=z\,y,\quad x\in(0,\infty),\quad \gamma\in\R,
\end{align}
that is $r(x) = e^{-x} x^{\gamma-1}$, $p(x) = x^{\gamma}e^{-x}$ and $q(x) \equiv 0$. Two special solutions for $z = 0$ are given by
\begin{align}
    u_{\infty, \gamma}(0, x) = 1, \quad v_{\infty, \gamma}(0,x) = \int_c^x t^{-\gamma} e^t dt = x^{-\gamma} e^x(1 + o(1)). 
    \end{align}
It follows that $\int_0^x u_{\infty, \gamma}(0, t)v_{\infty, \gamma}(0,t) r(t) dt \sim \int_0^x t^{-1} dt \sim \ln\, x$, as $x \to \infty$. This implies via \eqref{TraceFormula} (adopted to the right endpoint) that
\begin{align}\label{Zeta1Lag}
    \zeta_{Lag}(1; (c,x)) = \ln\, x + O(1).
\end{align}
As the corresponding resolvents are known to be Hilbert--Schmidt (the eigenvalues grow linearly), we can therefore conclude that
\begin{align}
    \varphi_{\infty, \gamma}(\lambda,x) \propto \exp\lbrace \lambda \ln\, x+O(1) \rbrace \propto x^\lambda, \quad x \to \infty,
\end{align}
and
\begin{align}
    \theta_{\infty, \gamma}(\lambda,x) \propto x^{-\gamma}\exp\lbrace x-\lambda \ln\, x+O(1) \rbrace \propto x^{-\lambda-\gamma} e^x, \quad x \to \infty.
\end{align}
This example provides a simple check for our results, as we can readily compare it with the exact solutions known from the literature (see \cite[Eqs.~13.2.6 and 13.2.26]{DLMF})
\begin{align}
U(-\lambda,\gamma;x) \underset{x\to\infty}{\propto} x^\lambda, \quad
x^{1-\gamma} {}_1F_1(1-\gamma-\lambda, 2-\gamma; x)\underset{x\to\infty}{\propto} x^{-\lambda-\gamma}e^x ,\quad (\lambda+\gamma)\notin \N,
\end{align}
where $U$ denotes Tricomi's confluent hypergeometric function (not to be confused with the parabolic cylinder function $U$ from the previous example). For the asymptotic behavior see \cite[Eqs.~13.7.1 and 13.7.3]{DLMF}. In particular, setting $\varphi_{\infty, \gamma}(z,x) = U(-z, \gamma; x)$ we see that
\begin{align}
    \frac{\varphi_{\infty, \gamma}(z_1, x) \cdots \varphi_{\infty, \gamma}(z_n, x)}{\varphi_{\infty, \gamma}(w_1, x) \cdots \varphi_{\infty, \gamma}(w_n, x)} \propto x^{z_1 + \dots + z_n - w_1 - \dots - w_n}, \quad x \to \infty,
\end{align}
which is $\propto 1$ if and only if $\sum_{j=1}^n z_j = \sum_{j=1}^n w_j$. Hence with the help of Corollary \ref{Coriii} we see that $U(-z, \a; x)$ is an entire function of minimal order $1$ in $z$. 

To write down a formula for the characteristic function of minimal order, we need to know the behavior of the principal and nonprincipal solutions near $x = 0$. Note that two solutions of $\tau_{Lag, \gamma} y = 0$ are given, similarly to before, by $y_1(x) = 1$ and $y_2(x) = \int_x^c t^{-\gamma} e^t \, dt \propto t^{-\gamma + 1}$ as $x \downarrow 0$ if $\gamma > 1$, and $\propto \ln\, x$ if $\gamma = 1$. In case $\gamma < 1$ we set $y_2(x) = \int_0^x t^{-\gamma} e^t \, dt \propto t^{-\gamma + 1}$. In particular, we see that $y_1(x)$ is principal for $\gamma \geq 1$ and nonprincipal for $\gamma < 1$ (for $y_2(x)$ the reverse is true).  For $x \downarrow 0$, we have $\varphi_{0, \gamma}(0,x)\theta_{0, \gamma}(0, x)r(x) \propto y_1(x)y_2(x)r(x) \propto 1$ for $\gamma \not = 1$, and $\propto \ln\, x$ for $\gamma = 1$. Importantly, $\varphi_{0, \gamma}(0,x)\theta_{0, \gamma}(0, x)r(x)$ is clearly integrable near a neighbourhood of $x = 0$. Thus by \cite[Thm.~4.2]{PS24}, the resolvents of the self-adjoint realizations of $\tau_{Lag, \gamma}|_{(0,c)}$ will be trace class and the behavior of principal and nonprincipal solutions will be the same irrespectively, of $z \in \bbC$:
\begin{align}
    \varphi_{0, \gamma}(z,x) \propto \begin{cases}
    1, & \gamma \geq 1,
\\
 x^{-\gamma + 1}, & \gamma < 1,
\end{cases} \quad \theta_{0, \gamma}(z,x) \propto 
    \begin{cases}
    x^{-\gamma + 1}, & \gamma > 1,
\\
\ln\, x , & \gamma = 1,
\\
1, & \gamma < 1.
    \end{cases}
\end{align}
Thus a characteristic function of the Friedrichs realization of minimal order $1$ is given by
\begin{align}\label{Fgamma}
    F_{\gamma, 0}(z) =
    \begin{cases}
    \lim_{x \downarrow 0} \cfrac{U(-z, \gamma; x)}{x^{-\gamma + 1}}=\dfrac{\Gamma(\gamma-1)}{\Gamma(-z)}, & \gamma > 1,
    \vspace*{7pt}
    \\
    \lim_{x \downarrow 0} \cfrac{U(-z, \gamma; x)}{\ln\, x}=-\dfrac{1}{\Gamma(-z)}, & \gamma = 1,
    \vspace*{7pt}
    \\
    \lim_{x \downarrow 0} U(-z, \gamma; x)=\dfrac{\Gamma(1-\gamma)}{\Gamma(-z-\gamma+1)}, & \gamma < 1,
    \end{cases}
\end{align}
where we have utilized \cite[\S 13.2(iii)]{DLMF}. As characteristic functions of minimal order $1$ are unique up to a gauge transformation of the form $\widetilde F_{\gamma, 0}(z) = e^{A+Bz} F_{\gamma, 0}(z)$, we can alternatively choose $\widetilde F_{\gamma, 0}(z) =1/\Gamma(-z)$ for $\gamma\geq1$ and $\widetilde F_{\gamma, 0}(z) =\Gamma(-z-\gamma+1)$ for $\gamma<1$, which is also continuous with respect to $\gamma$. That it is independent of $\gamma$ for $\gamma \geq 1$ is a consequence of the fact that the Laguerre polynomials $L^{(\gamma)}_n(x)$ satisfy $\tau_{Lag, \gamma} L^{(\gamma)}_n = n L^{(\gamma)}_n$, that is, the eigenvalues do not depend on $\gamma$. The Laguerre polynomials are principal at $x = 0$, and hence are the Friedrichs eigenfunctions, if and only if $\gamma \geq 1$.

From the behavior of the nonprincipal solution near $x = 0$ it is clear that $\tau_{Lag, \gamma}$ is in the limit circle case at $x = 0$ if and only if $\gamma \in (0,2)$. In this case one can obtain the characteristic function corresponding to the $\alpha$-boundary condition at $x = 0$ via
\begin{align}
    F_{\gamma, \alpha}(z) = \cos(\a)\wti U(-z,\gamma; 0)+\sin(\a)\wti U^{\prime}(-z,\gamma; 0).
\end{align}
The above can be evaluated explicitly using \cite[Sect. 13]{DLMF}; we leave the details to the reader.

\appendix

\section{Asymptotical integral formula for \texorpdfstring{$\zeta(1; (x,c))$}{zeta(1;(x,c))} as \texorpdfstring{$x \downarrow a$}{x to a}}\label{appendix}
The goal of this appendix is to prove the following (asymptotic) integral formula for $\zeta(1; (x,c))$: 
\begin{align}\label{IntRep}
    \int_x^c u_a(\lambda,t) v_a(\lambda,t) r(t)dx = \zeta(1; (x,c)) + O(\zeta(2;(x,c))), \quad \text{for} \ x \downarrow a,
\end{align}
where $W(v_a(\lambda, \, \cdot \,), u_a(\lambda, \, \cdot \,)) = 1$. The above formula is very similar to the \emph{exact} integral representation in Proposition \ref{Prop:ExactZeta} for $\zeta(1; (x,c))$. Though not exact, \eqref{IntRep} is very convenient in practice, and we use it in Examples \ref{Sect:PowerPotentials}, \ref{LagEx} (see \eqref{Zeta1Harm}, \eqref{eq:zeta1}, \eqref{Zeta1Lag}). Moreover, it suffices to determine the exact spectral dependence of the principal and nonprincipal solutions in the Hilbert--Schmidt resolvent case.

\begin{lemma}\label{LemmaInt1}
    Assume that the self-adjoint realizations of $\tau|_{(a,c)}$ have a purely discrete spectrum. Let $u_a(\lambda,x)$ ($v_a(\lambda,x)$) be a principal (nonprincipal) solution at $x = a$ with $\lambda \in \mathbb R$, without any additional normalization assumption. Then
    \begin{align}\label{EqInt1}
        \lim_{x \downarrow a} \frac{u_a(\lambda,x)}{v_a(\lambda,x)} \int_x^c v_a^2(\lambda,t) r(t) dt = 0. 
    \end{align}
    Moreover, if additionally $W(v_a(\lambda, \cdot), u_a(\lambda, \cdot)) \equiv 1$, then we have
    \begin{align}\label{TraceFormula}
        \int_x^c u_a(\lambda,t) v_a(\lambda,t) r(t)dx = \zeta(1; (x,c)) + O(\zeta(2;(x,c)), \quad \text{for} \ x \downarrow a,
    \end{align}
    and in case $\lambda = 0$ we even have
    \begin{align}\label{Lambda0Formula}
        \int_x^c u_a(0,t) v_a(0,t) r(t)dx = \zeta(1; (x,c)) + O(1), \quad \text{for} \ x \downarrow a.
    \end{align}
\end{lemma}
Note that the condition on the restriction $\tau|_{(a,c)}$ in Lemma \ref{LemmaInt1} is identical to the discrete spectrum assumption made in \cite[Lem.~3.2, 3.3]{GZ06}, which, in our setting, is equivalent to the existence of an entire fundamental system $\varphi_a(z,x)$, $\theta_a(z,x)$ (without any particular normalization), see \cite[Lem.~2.2, 2.4]{KST_IMRN}.

The key to the proof of Lemma \ref{LemmaInt1} given below will be the classic trace formula expressing the trace of a self-adjoint realization of $\tau|_{(c,d)}$ via the integral of the diagonal of the corresponding Green's function. To formulate this argument, we need to introduce some additional notation. We refer to the \emph{$\psi(x)|_{x = c}$-boundary condition} as the boundary condition satisfied by $\psi(x)$ at $x = c$. For example, $f$ will satisfy the $\psi(x)|_{x = c}$-boundary condition if and only if
\begin{align}
    \cos(\alpha) f(c) + \sin(\alpha) f^{[1]}(c) = 0,
\end{align}
where $\cos(\alpha) \psi(c) + \sin(\alpha) \psi^{[1]}(c) = 0$. Here it is implicitly assumed that $\psi$ is sufficiently regular for the above to make sense. We will write $\lambda_j(\psi_1(c), \psi_2(d))$ for the $j^{th}$ eigenvalue of the self-adjoint realization of $\tau|_{(c,d)}$ with $\psi_1(x)|_{x=c}$-boundary conditions at $x = c$ and $\psi_2(x)|_{x=d}$-boundary condition at $x = d$. A more accurate, though vastly more cumbersome, notation would be $\lambda_j(\psi_1(c), \psi^{[1]}_1(c); \psi_2(d), \psi_2^{[1]}(d))$. 

Now assuming $W(v_a(\lambda, \, \cdot \, ),  u_a(\lambda, \, \cdot \, )) \equiv 1$, we have the identity
\begin{align}\label{TraceCD}
    \int_{x}^c u_a(\lambda,t)v_a(\lambda,t)r(t) dt = \sum_{j \geq 1} \frac{1}{\lambda_j(u_a(\lambda,x), v_a(\lambda,c))-\lambda}, \quad \lambda \in \bbR.
\end{align}
The above equality is a consequence of the fact that $u_a(\lambda,t)$ satisfies the $u_a(\lambda, t)|_{t=x}$-boundary condition at $x$ and analogously for $v_a(\lambda, t)$. From the Wronskian normalization it follows that $u_a(\lambda,t)v_a(\lambda,t)$ is the diagonal of the corresponding Green function and \eqref{TraceCD} follows from Mercer's Theorem. Note that $\lambda_j(u_a(\lambda,x), v_a(\lambda,c)) = \lambda$ cannot happen, as otherwise $u_a(\lambda,x)$ would satisfy the same boundary condition as $v_a(\lambda,x)$ at $x = c$, which is impossible since $u_a$ and $v_a$ are linearly independent for all $z$.
\begin{proof}[Proof of Lemma \ref{LemmaInt1}]
As the proof is rather technical, we first provide an informal overview. The proof is divided into two parts: The first part uses the trace formula involving the diagonal of the Green's function to reduce Lemma \ref{LemmaInt1} to the convergence of $\lambda_j(u_a(\lambda,x), v_a(\lambda,c)) \to \lambda_{j, F}(a,c)$ as $x \downarrow a$ in the case $v_a(\lambda, c) = 0$ (made for simplicity), where $\lambda_{j, F}(a,c)$ denotes the Friedrichs eigenvalues corresponding to $\tau|_{(a,c)}$. If we could replace $\lambda_j(u_a(\lambda,x), v_a(\lambda,c))$ by the $j^{th}$ Dirichlet eigenvalue $\lambda_{j, D}(x,c)$ this would be a standard result, see \cite[Ch.~10.8]{Ze05}. Hence in the second part of the proof, we want to show that the $(\lambda, x)$-dependence of the left $u_a(\lambda,t)|_{t=x}$-boundary condition at $x$ actually improves the convergence rate. This is why the auxiliary function $\psi_{j,x}$ is introduced, which turns out to be a Dirichlet eigenfunction for $\tau|_{(x^*, c)}$ and $x^{*} \in (a,x)$.

\textbf{First part}: Throughout the proof let $W(v_a(\lambda, \, \cdot \, ),  u_a(\lambda, \, \cdot \, )) \equiv 1$.  We can w.l.o.g assume that $v_a(\lambda,c) = 0$ and $v_a(\lambda,x) > 0$ for $x\in (a,c)$ (note that, ultimately, only the behavior of nonprincipal solutions as $x \downarrow a$ matters, and that it is the same over all choices). In this case $v_a(\lambda,x)$ would satisfy the Dirichlet boundary condition at $x = c$. We can now compare the above integral formula with
\begin{align}\label{TraceDiri}
    \int_x^c \Big[u_a(\lambda,t) - \frac{u_a(\lambda,x)}{v_a(\lambda,x)} v_a(\lambda,t)\Big] v_a(\lambda,t) r(t) dt = \sum_{j \geq 1} \frac{1}{\lambda_{j, D}(x,c) -\lambda} = \zeta(1; (x,c)) + O(\lambda \zeta(2; (x,c))),
\end{align}
 where $\lambda_{j, D}(x,c)$ is shorthand for the Dirichlet eigenvalues corresponding to $\tau|_{(x,c)}$ (for the error we have used $\frac{1}{\lambda_j -\lambda} \approx \frac{1}{\lambda_j} + \frac{\lambda}{\lambda_j^2} + O(\frac{\lambda^2}{\lambda_j^3})$ for $\lambda_j \to \infty$; if $\lambda = 0$ then we have an exact formula). 

Now we know that $\lambda_1(u_a(\lambda,x), v_a(\lambda,c)) < \lambda_{1,D}(x,c) < \lambda_2(u_a(\lambda,x), \theta(\lambda,c)) < \dots$ (note as $v_a(\lambda, c) = 0$, only the boundary condition at $x$ differs for the two problems). Thus we have
\begin{align}\nonumber
    \frac{u_a(\lambda,x)}{v_a(\lambda,x)} &\int_x^c v_a^2(\lambda,t) r(t) dt = \sum_{j\geq 1} \frac{1}{\lambda_j(u_a(\lambda,x), v_a(\lambda,c)) - \lambda} - \sum_{j \geq 1} \frac{1}{\lambda_{j, D}(x,c) - \lambda}
    \\\label{leibniz}
    &= \sum_{j = 1}^N \Bigg[\frac{1}{\lambda_j(u_a(\lambda,x), v_a(\lambda,c)) - \lambda} - \frac{1}{\lambda_{j, D}(x,c) - \lambda}\Bigg] + Er_{N+1},
\end{align}
with $|Er_{N+1}| \leq \frac{1}{\lambda_{N+1}(u_a(\lambda,x), v_a(\lambda,c)) - \lambda}$ by the classic alternating series estimates. Note that $\lambda_{j,D}(x,c) \to \lambda_{j, F}(a,c)$ as $x \downarrow a$, where $\lambda_{j, F}(a,c)$ denotes the Friedrichs eigenvalues corresponding to $\tau|_{(a,c)}$, see \cite{BEWZ}, \cite[Thm.~10.8.2]{Ze05}. If we can show that $\lambda_j(u_a(\lambda,x), v_a(\lambda,c)) \to \lambda_{j, F}(a,c)$ as $x \downarrow a$, then this would imply that the sum of $N$ terms will converge to $0$, while the error term can be bounded by  $1/\lambda_{N+1, F}(a,c)$ as $x \downarrow a$, which can be made arbitrary small by choosing $N$ large enough, implying that \eqref{EqInt1} holds. From \eqref{TraceDiri} it would then also follow that $\int_x^c u_a(\lambda,t) v_a(\lambda,t) r(t)dx = \zeta(1; (x,c)) + O(\zeta(2;(x,c))$.

\textbf{Second  part}: It remains to show that $\lambda_j(u_a(\lambda,x), v_a(\lambda,c)) \to \lambda_{j, F}(a,c)$ as $x \downarrow a$. Let us make the following observation: Note that trivially, $\lambda = \lambda_1(L^2_a, u_a(\lambda,c))$, where $\lambda_j(L^2_a, u_a(\lambda,c))$ denotes the $j^{th}$ eigenvalue corresponding to $\tau|_{(a,c)}$ with $u_a(\lambda,x)|_{x = c}$-boundary condition at $x = c$, and ``$L^2$-boundary condition" at $x = a$. This is immediate as $u_a(\lambda, \, \cdot \,)$ is an eigenfunction which is nonzero on $(a,c]$ (as $v_a(\lambda, x)$ does not vanish there). 

Now denote by $\psi_{j,x}$ the $j^{th}$ eigenfunction corresponding to $\lambda_j(u_a(\lambda,x), v_a(\lambda,c))$ (recall $v_a(\lambda, c) = 0$). As $u_a(\lambda,x)$ has no zeros on $(a,c]$, it follows that $\lambda_j(u_a(\lambda,x), v_a(\lambda,c)) > \lambda$ for all $j \in \bbN$. 

Next, continue $\psi_{j,c}(t)$ to $(a,x]$ such that $\tau \psi_{j,c} = \lambda_j(u_a(\lambda,x), v_a(\lambda,c))\psi_{j,c}$ remains valid. Recall that $\psi_{j,c}(t)$ and $u_a(\lambda, t)$ satisfy, by assumption, the \emph{same} boundary condition  at $t = x$, however $\lambda_j(u_a(\lambda,x), v_a(\lambda,c)) > \lambda$. As $u_a(\lambda, t)$ is principal at $t = a$, it follows that $\psi_{j,c}(t)$ must have a (largest) zero $x^* \in (a,x)$ and consequently $\lambda_j(u_a(\lambda,x), v_a(\lambda,c))$ is equal to the $j^{th}$ Dirichlet eigenvalue $\lambda_{D,j}(x^*, c)$. Hence, as $x \downarrow a$, we obtain that $\lambda_j(u_a(\lambda,x), v_a(\lambda,c)) = \lambda_{D,j}(x^*, c) \to \lambda_{F, j}(a,c)$, finishing the proof.  
\end{proof}
Let us also remark that, by using similar techniques, an analog of \eqref{EqInt1} can be shown with the roles of the principal and nonprincipal solutions interchanged (though we will not show it here). 
\begin{lemma}\label{LemmaInt2}
    Assume that the self-adjoint realizations of $\tau|_{(a,c)}$ have a purely discrete spectrum. Let $u_a(\lambda,x)$ ($v_a(\lambda,x)$) be a principal (nonprincipal) solution at $x = a$ with $\lambda \in \mathbb R$, without any additional normalization assumption. Then
    \begin{align}\label{EqInt2}
        \lim_{x \downarrow a} \frac{v_a(\lambda,x)}{u_a(\lambda,x)} \int_a^x u_a^2(\lambda,t) r(t) dt = 0.
    \end{align}
\end{lemma}

The conclusions of Lemmas \ref{LemmaInt1}, \ref{LemmaInt2} do not hold if the spectrum has a continuous part. Consider the example of $\tau = -\frac{d^2}{dx^2}$, $x \in (0, \infty)$ and choose $v_{\infty}(\lambda, x) = \frac{e^{\sqrt{-\lambda}x}}{2 \sqrt{-\lambda}}$, $u_{\infty}(\lambda, x) = e^{-\sqrt{-\lambda}x}$ (we have switched to the right endpoint, but the same logic applies). Then one can compute
\begin{align}
     \lim_{x\to \infty} \frac{u_{\infty}(\lambda, x)}{v_{\infty}(\lambda, x)} \int_0^x v_{\infty}^2(\lambda, t)dt =\lim_{x\to \infty} \frac{v_{\infty}(\lambda, x)}{u_{\infty}(\lambda, x)} \int_{x}^\infty u_{\infty}^2(\lambda, t)dt= \frac{1}{4(-\lambda)} = \frac{1}{4 \dist(\lambda, \sigma_{ess}(\tau))}
\end{align}
($\sigma_{ess}(\tau)$ is the essential spectrum of any self-adjoint realization of $\tau$). This raises the following problem: 

\begin{problem}\label{Prob:Ess}
What else can be said about the limits in \eqref{EqInt1}, \eqref{EqInt2} under the assumption that the essential spectrum is present, $\lambda < \sigma_{ess}(\tau|_{(a,c)})$, and $W(v_a(\lambda, \, \cdot \,), u_a(\lambda, \, \cdot \,)) = 1$ (or $-1$ if we consider the right endpoint). Do the limits exist and, if so, are they always equal to $\frac{1}{4 \dist(\lambda, \sigma_{ess}(\tau|_{(a,c)}))}$?
\end{problem}

%%%%%%%%%%%%%%%%%%%%%%%%%%%%%%%%%%%%%

\noindent 
{\bf Acknowledgments.}
M.P. was supported by the starting grant from the Ragnar S\"oderbergs Foundation. J.S. was supported in part by an AMS--Simons Travel Grant.

\end{document}